\documentclass[12pt,onecolumn]{IEEEtran}
\newcommand{\Pp}{\mathbb{P}}
\usepackage[utf8]{inputenc}

\usepackage{amsmath,amssymb,amsthm}
\usepackage{graphicx}
\usepackage{float}
\usepackage{url}
\usepackage{tikz}
\usetikzlibrary{quotes}
\usepackage{xcolor}

\usetikzlibrary{arrows.meta,matrix,positioning}

\usepackage{hyperref}
\hypersetup{hidelinks,pdftitle={Coin flipping and waiting times paradoxes: Why fair coins are exceptional},pdfauthor={Søren Riis and Mike Paterson}}
\usepackage{longtable}
\newcommand{\E}{\mathbb{E}}

\newtheorem{theorem}{Theorem}
\newtheorem{lemma}[theorem]{Lemma}
\newtheorem{proposition}[theorem]{Proposition}
\newtheorem{corollary}[theorem]{Corollary}

\theoremstyle{definition}
\newtheorem{definition}[theorem]{Definition}
\newtheorem{example}[theorem]{Example}
\newtheorem{remark}[theorem]{Remark}
\newtheorem{openquestion}[theorem]{Open Question}
\newcommand{\sdom}{\succ_{\mathrm{st}}}

\newif\ifarxivversion
\arxivversiontrue

\title{Coin flipping and waiting times paradoxes \\ 
Why fair coins are exceptional}

\author{S{\o}ren Riis\IEEEauthorrefmark{1}\hspace{1em}and\hspace{1em}Mike Paterson\IEEEauthorrefmark{2}%
\thanks{\IEEEauthorrefmark{1} Centre for Fundamentals of AI and Computational Theory, Queen Mary University of London, London, United Kingdom.}%
\thanks{\IEEEauthorrefmark{2}  DIMAP and the Department of Computer Science,  University of Warwick, Coventry, United Kingdom.}}

\begin{document}
\maketitle

\begin{abstract}
Penney’s Ante exhibits non‑transitivity when two target strings race to appear in a \emph{shared} stream of coin tosses. 
We study instead \emph{independent} string races, where each player observes their own \emph{independent and identically distributed} (i.i.d.) coin/die stream (possibly biased), and the winner is the player whose target appears first (under an explicit tie convention).
We derive compact generating-function formulas for the waiting time $W_U$ until a target string $U$ first appears, and a Hadamard--generating-function calculus for head-to-head odds.
Blom proved that for a fair $s$-sided die, expected waiting-time order agrees
with stochastic order on all finite patterns.  We give a self-contained
coefficientwise proof and deduce that, in every independent two-player race
with random tie-break, the pattern with the shorter expected waiting time is
the favourite.
Our principal new theorem is the Bernoulli$(p)$ characterisation: universal
agreement between mean waiting time and head-to-head odds holds if and only
if $p=\tfrac12$.

  For every biased coin with $p\in(0,1)\setminus\{\tfrac12\}$, there are finite patterns $U,V$
such that $\mathbb E[W_U]<\mathbb E[W_V]$ but $U$ has random-tie score less
than $1/2$ against $V$.
The proof of this biased converse combines a uniform analytic construction near $p=\tfrac12$ with exact rational certificates on the remaining parameter range.  For some biased coins there are also directed cycles in pairwise head-to-head preferences; we give explicit examples and certified computational classifications for short patterns.

\end{abstract}

\section{Introduction}

\paragraph*{Brief history and variants}
The game now known as \emph{Penney's Ante} was introduced by Walter Penney in the late 1960s and became widely known through Martin Gardner's \emph{Mathematical Games} column. The classical setting has two players first choosing binary strings (over \(\{\mathrm{H},\mathrm{T}\}\)) of the same length; a single infinite sequence of shared coin tosses is then generated, and the winner is the player whose string occurs first as a contiguous block. The surprising feature is \emph{non-transitivity}: for many choices of the first player's string, the second player can choose a counterstring that is more likely to appear earlier, and these preferences can cycle.

This phenomenon attracted sustained attention from John H.\ Conway and others, connecting to combinatorics on words, automata, and probability. Beyond the shared-coin model, variants have been considered where players do not condition on the same random sequence. In particular, one can let players use \emph{independent} sources of randomness, maybe \emph{different coins}, possibly with different biases. 
Nickerson surveys Penney's game and its extensions \cite{nickerson2007penney}. We formalise and analyse independent-source models below. Classical treatments of runs and first‑occurrence phenomena in i.i.d.\ sequences can already be found in Feller \cite{Feller1968,Feller1971}.

Analytically, the classic shared-source game has been studied via string overlaps (borders) and generating functions since Guibas--Odlyzko \cite{Guibas1981}, with earlier popular accounts by Penney and Gardner \cite{Penney1969,Gardner1974}. Elizalde and Lin proposed a \emph{permutation} analogue in which the targets are order patterns observed in a \emph{shared} stream of i.i.d.\ real values; they show the game remains non-transitive in that setting~\cite{Elizalde2024}. The present paper takes a complementary direction: we keep the objects as words but fundamentally change the \emph{information structure}---\textbf{each player has their own independent coin or die}, possibly with bias. This shift removes the shared-source coupling of first-hit times and leads to a different structural picture. Blom's fair-source theorem states that stochastic dominance induces a \emph{total pre-order} on all strings (even across lengths), ordered by expected waiting time~\cite{Blom1984}; we give a self-contained proof and draw the independent-race consequence in Theorem~\ref{thm:fair-total-order}.  Our main new structural result is the binary converse (Theorem~\ref{thm:coin-mean-race-iff}): every non-fair coin bias admits a two-pattern reversal. Methodologically, we replace large joint Markov chains with a compact Hadamard--generating-function calculus for head-to-head odds, and cleanly separate strict wins from tie conventions.

\begin{theorem}[Blom's fair-source ordering theorem]
For a fair $s$-sided die, stochastic dominance induces a \emph{total pre-order} on all strings (even across lengths), ordered by expected waiting time.
\label{thm:fair-total-order}
\end{theorem}

\begin{theorem}[Mean--race agreement characterises a fair coin]
\label{thm:coin-mean-race-iff}
Let $p\in(0,1)$ and let two players use independent Bernoulli$(p)$ streams, with a random tie-break.  The following are equivalent:
\begin{enumerate}
\item $p=\tfrac12$;
\item for every two nonempty binary patterns $U,V$,
\[
\mathbb E[W_U]<\mathbb E[W_V]
\quad\Longrightarrow\quad
\mathbb P(W_U<W_V)+\tfrac12\mathbb P(W_U=W_V)>\tfrac12.
\]
\end{enumerate}
For every $p\neq\tfrac12$ the failure is strict: some finite $U,V$ satisfy the mean inequality above but have random-tie score strictly below $1/2$.
\end{theorem}

 Unless stated otherwise, we assume each player's coin has the \emph{same} bias \(p\) for heads, although we also discuss more general asymmetric cases. Because the sources are independent, the waiting-time distributions for different players no longer interfere through a shared source; this fundamentally changes how comparisons of expected waiting times relate to head-to-head winning probabilities. In particular, it opens the door to new paradoxes (and removes others) relative to the shared-coin Penney framework.
We use the term ``paradox'' for two phenomena: (i) a reversal between expected waiting time and head-to-head win probability, and (ii) a non-transitive cycle in the head-to-head preference digraph.

\medskip
To get our heads around this problem, let us first make some rather trivial observations. For example, suppose one player is using an ordinary six-sided die while the other is using a one-sided coin (always showing ``T''). The winning sequence for the player with the die is ``6'', while it is ``TTTT'' for the player with the coin. Assume that ties count in favour of the die player. This example is a variant of the famous betting game of the Chevalier de M\'er\'e (1607--1684), analysed by Blaise Pascal. The sequence ``TTTT'' occurs deterministically on the fourth toss, whereas the die produces the winning outcome ``6'' after six tosses on average.  Despite this, the die player is more likely to win the game.
This kind of paradoxical comparison between mean time and win odds is a recurring theme in classical sources; see, e.g., \cite[Ch.~XIII]{Feller1968}.

\section{Average waiting times and generating functions}
Consider an \(s\)-sided coin (or die) with sides \(1,2,\ldots ,s\) occurring with probabilities \(p_1,p_2,\ldots ,p_s\) per toss (independently across time). Unless explicitly stated otherwise, $p_i>0$ for every face and $s\ge2$.

\smallskip
\noindent\textbf{String probabilities.} For a word \(T=t_1\cdots t_L\) we write \(\Pr(T)=p_{t_1}\cdots p_{t_L}\).

\smallskip
\noindent\textbf{Fairness.} We call the source \emph{fair} if it is uniform, i.e.\ $p_i=1/s$ for all $i=1,\dots,s$ (so for a coin, $p=1/2$).

\begin{definition}[Borders and prefix--suffix lengths]\label{def:borders}
Let \(T\) be a finite string of length \(L\). A \emph{border} of \(T\) is a substring that is both a prefix and a suffix of \(T\). We refer to the lengths \(\ell\in\{1,\dots,L\}\) of these borders as the \emph{prefix--suffix lengths} of \(T\); formally, these are the lengths \(\ell\) such that:
\[
T[1..\ell] \;=\; T[L-\ell+1..L].
\]
We write \(\mathcal{B}(T)\) for the set of all such \(\ell\). (Some authors reserve “proper border’’ for \(\ell<L\); in this paper we \emph{include} \(\ell=L\).) The \emph{border polynomial} is
\[
H_T(y)\;:=\;\sum_{\ell\in\mathcal{B}(T)} y^\ell .
\]
\end{definition}

\begin{example}[Quick examples]
Over \(\{\mathrm{H},\mathrm{T}\}\):
\(\mathcal{B}(\mathrm{HH})=\{1,2\}\),\ 
\(\mathcal{B}(\mathrm{HT})=\{2\}\),\ 
\(\mathcal{B}(\mathrm{TH})=\{2\}\),\
\(\mathcal{B}(\mathrm{TT})=\{1,2\}\),\
\(\mathcal{B}(\mathrm{HTHT})=\{2,4\}\).
\end{example}

Let \(X_1,X_2,\ldots\) be i.i.d.\ with \(\Pr(X_n=i)=p_i\) on \(\{1,2,\ldots,s\}\).

For a target string \(T\) of length \(L\), define the waiting time (first-hit time)
\[
W_T:=\min\{n\ge L:\; (X_{n-L+1},\ldots,X_n)=T\},
\]
and let \(A_T(x)\) be the probability generating function (pgf) of the waiting time, i.e.
\(A_T(x)=\mathbb{E}[x^{W_T}]=\sum_{n=0}^\infty a_{T,n} x^n\) where \(a_{T,n}=\Pr(W_T=n)\) is the probability that \(T\) first occurs after exactly \(n\) tosses.

\begin{theorem}\label{th1}
Let \(\mathcal{B}(T)=\{\ell_1<\ell_2<\cdots<\ell_r\}\) and set \(T_j:=T[1..\ell_j]\) for \(j=1,\ldots,r\).
The generating function for the waiting time is given by
\[
A_T(x)\;=\;\frac{1}{\,1+(1-x)\,h_T(x)\,} ,
\]
where
\[
h_T(x)= \frac{1}{\Pr(T_1)\, x^{\ell_1}} +  \frac{1}{\Pr(T_2)\, x^{\ell_2}} + \cdots +  \frac{1}{\Pr(T_r)\, x^{\ell_r}}.
\]
In particular, the expected waiting time is
\[
 \mathbb{E}[W_T] \;=\; \frac{1}{\Pr(T_1)} + \frac{1}{\Pr(T_2)} + \cdots +\frac{1}{\Pr(T_r)}.
\]
\end{theorem}

\begin{remark}[Rewriting \(h_T\) and the expectation]\label{fair}
With this notation, the function \(h_T(x)\) in Theorem~\ref{th1} can be written as
\[
h_T(x)\;=\;\sum_{\ell\in\mathcal{B}(T)} \frac{1}{\Pr(T[1..\ell])\,x^\ell},
\]
and in the \emph{fair \(s\)-sided} case we have
\(
h_T(x)=\sum_{\ell\in\mathcal{B}(T)} \big(\tfrac{s}{x}\big)^{\!\ell}
\)
and
\(
\mathbb{E}[W_T]=H_T(s)=\sum_{\ell\in\mathcal{B}(T)} s^\ell
\).
\end{remark}
\begin{remark}[Provenance]
Classical derivations of single‑pattern waiting‑time pgf’s and means via renewal arguments already appear in Feller \cite{Feller1968,Feller1971}; concise martingale/optional‑stopping proofs are given by Li \cite{Li1980}, while Markov‑chain (hitting‑time) embeddings are developed in \cite{GerberLi1981}. The finite‑Markov‑chain (DFA) methodology and related gambling‑teams viewpoint provide alternative routes to the same formulas and their extensions \cite{FuKoutras1994,PozdnyakovEtAl2006}.
\end{remark}

\begin{corollary}\label{col01}
Assume the coin (die) is \(s\)-sided and each side occurs with probability \(1/s\). Then
\[
A_T(x)= \frac{1}{1+(1-x) \left(\left(\tfrac{s}{x}\right)^{\ell_1} +  \left(\tfrac{s}{x}\right)^{\ell_2} + \cdots +  \left(\tfrac{s}{x}\right)^{\ell_r}\right)}.
\]
In particular, \(\mathbb{E}[W_T]=s^{\ell_1} + s^{\ell_2}+ \cdots + s^{\ell_r}\).
\end{corollary}

Analogous waiting‑time expressions and reductions hold for multistate (non‑binary) trials; see \cite{Antzoulakos2001}.

\begin{corollary}
For a fair \(s\)-sided die, \(\mathbb{E}[W_T]\) is a multiple of \(s\) and uniquely determines the set of prefix–suffix lengths. In turn, the prefix–suffix lengths uniquely determine \(A_T(x)\) and thus the full waiting‑time distribution.
\end{corollary}

\begin{proof}
In the fair case, \(\mathbb{E}[W_T]=\sum_{\ell\in\mathcal{B}(T)} s^\ell\), which is divisible by \(s\). Moreover, base‑\(s\) expansion is unique, so the representation as a sum of \emph{distinct} powers of \(s\) uniquely determines the subset \(\mathcal{B}(T)\subseteq\{1,\ldots,|T|\}\). By Corollary~\ref{col01}, \(\mathcal{B}(T)\) determines \(A_T(x)\).
\end{proof}

\begin{example}
The string \(T=\text{``HTHT''}\) has prefix–suffix lengths \(2\) and \(4\); for a fair coin \(\mathbb{E}[W_T]=2^2+2^4=20\).
The string \(T=\text{``13166131''}\) has prefix–suffix lengths \(1,3,8\); for a fair six‑sided die,
\(\mathbb{E}[W_T]=6^1+6^3+6^8=1{,}679{,}838\).
\end{example}

\section{Generating functions, stochastic dominance and paradoxes}\label{sec:stoch-dom}

\subsection{Setup and notation}\label{subsec:stoch-dom-setup}
We continue with the i.i.d.\ \(s\)-sided die and the notation from the previous section. Let \(\mathcal{A}=\{1,2,\ldots,s\}\) be the alphabet of die faces, with \(\Pr(X_n=a)=p_a\) for \(a\in\mathcal{A}\). Given two strings \(T_1\) and \(T_2\), we say that \(W_{T_1}\) \emph{statistically (stochastically) dominates} \(W_{T_2}\) (write \(W_{T_1}\sdom W_{T_2}\)) if \(\Pr(W_{T_1}>n)\ge \Pr(W_{T_2}>n)\) for all \(n\ge 0\), with strict inequality for some \(n\). If \(\Pr(W_{T_1}>n)=\Pr(W_{T_2}>n)\) for all \(n\ge0\), we say that \(W_{T_1}\) and \(W_{T_2}\) are \emph{statistically equivalent}.

Throughout this section we use the representation from Theorem~\ref{th1}:
\begin{equation}\label{eq:main-c-form}
  A_T(x)\;=\;\frac{1}{\,1+(1-x)\,h_T(x)\,},\qquad 0<x<1,
\end{equation}
where \(h_T(x)\) is the pattern-dependent Laurent polynomial specified in Theorem~\ref{th1}. At $x=0$ the formula is interpreted after cancellation, by continuous extension.

We regard \eqref{eq:main-c-form} as an identity of formal power series in \(x\) (valid for any die with full support, i.e.\ \(p_a>0\) for all \(a\in\mathcal{A}\)). Explicitly,

\[
  h_T(x)\;=\;\sum_{\ell\in\mathcal{B}(T)} \frac{1}{\Pr\bigl(T[1..\ell]\bigr)}\,x^{-\ell},
\]
i.e. a finite Laurent polynomial in \(x^{-1}\) with positive coefficients. In the fair \(s\)-sided case this
simplifies to
\[
  h_T(x)\;=\;H_T\!\left(\frac{s}{x}\right)\;=\;\sum_{\ell\in\mathcal{B}(T)}
  \left(\frac{s}{x}\right)^{\!\ell},
\]
so the coefficients are nonnegative integers determined solely by the border structure. Whenever we
invoke coefficientwise comparisons of \(h\)-polynomials later on, we do so under this fair–die
specialisation (or we state the additional hypothesis explicitly).

\subsection{A basic algebraic lemma}
\begin{lemma}[Difference factorisation]\label{lem:factor}
Let \(U,V\) be two strings with generating functions \(A_U(x),A_V(x)\) and
polynomials \(h_U(x),h_V(x)\) as in Theorem~\ref{th1}. Then, for all \(x\in(0,1)\),
\[
  A_U(x)-A_V(x)\;=\;(1-x)\,\bigl(h_V(x)-h_U(x)\bigr)\,A_U(x)\,A_V(x).
\]
\end{lemma}
\begin{proof}
Compute
\[
  \frac{1}{A_V}-\frac{1}{A_U}=(1+(1-x)h_V)-(1+(1-x)h_U)
  =(1-x)(h_V-h_U).
\]
Multiplying by \(A_U(x)A_V(x)\) yields the claimed identity.
\end{proof}

It is often convenient to work with the \emph{tail} generating function
\(S_T(x):=\sum_{n\ge 0}s_{T,n}x^n\) 
where \(s_{T,n} = \Pr(W_T>n) 
= 1-\sum_{t=0}^{n} a_{T,t}\).
Since \(S_T(x)=\frac{1-A_T(x)}{1-x}\), Lemma~\ref{lem:factor} is equivalent to
\begin{equation}\label{eq:diff-with-S}
  S_V(x)-S_U(x)\;=\;\bigl(h_V(x)-h_U(x)\bigr)\,A_U(x)\,A_V(x).
\end{equation}

\begin{remark}[Mean waiting time]
Differentiating \eqref{eq:main-c-form} at \(x=1^{-}\) gives
\[
  A_T'(1)=\mathbb{E}[W_T]=h_T(1).
\]
Thus, ordering by \(\mathbb{E}[W_T]\) is the same as ordering by \(h_T(1)\).
\end{remark}

\subsection{Stochastic dominance versus expectation under a fair die}
For a fair die, the coefficients of \(h_T(x)\) admit a purely combinatorial description in terms
of the borders of \(T\); in particular, each coefficient is a nonnegative integer.

\begin{theorem}[Dominance implies larger mean]\label{thm:dom-implies-mean}
For any die (not necessarily fair) and any strings \(T_1,T_2\), if \(W_{T_1}\sdom W_{T_2}\) then
\(\mathbb{E}[W_{T_1}]>\mathbb{E}[W_{T_2}]\).
\end{theorem}
\begin{proof}
Stochastic dominance gives \(\Pr(W_{T_1}>n)\ge \Pr(W_{T_2}>n)\) for all \(n\ge0\), with strict inequality for some \(n\). Summing over \(n\ge0\) yields
\(\sum_{n\ge0}\Pr(W_{T_1}>n)>\sum_{n\ge0}\Pr(W_{T_2}>n)\), i.e.\ \(S_{T_1}(1)>S_{T_2}(1)\). But
\(S_T(1)=\sum_{n\ge0}\Pr(W_T>n)=\mathbb{E}[W_T]\).
\end{proof}

\begin{remark}[Dominance and mean]
\label{rem:dom-vs-mean}
For any source, stochastic dominance implies larger expectation (Theorem~\ref{thm:dom-implies-mean}).
Under a fair die we will later prove a converse: the mean waiting time totally orders patterns and
determines stochastic dominance; see Theorem~\ref{thm:fair-total-order} and
Corollary~\ref{cor:order-by-expectation-compact}.
\end{remark}

\subsection{Total preorders and equivalence}
Write \(\prec_{\mathrm{st}}\) for the stochastic-dominance strict partial order
on strings: \(T_1\prec_{\mathrm{st}}T_2\) iff \(W_{T_2}\sdom W_{T_1}\).
It is transitive and acyclic for every source. We call it \emph{total up to
equivalence} on a set of strings if every pair is comparable or has the same
waiting-time distribution. Dominance or equivalence then defines a total
preorder \(\preceq_{\mathrm{st}}\).

\begin{proposition}[Border-equivalence]\label{prop:border-equivalence}
If two strings \(T_1,T_2\) have the same multiset of borders with the same letter-counts at each border
length, then \(h_{T_1}(x)=h_{T_2}(x)\) and hence \(A_{T_1}(x)=A_{T_2}(x)\) and \(W_{T_1}\) and \(W_{T_2}\) are
identically distributed.
\end{proposition}
\begin{proof}
Immediate from the construction of \(h_T(x)\) in Theorem~\ref{th1}, which depends only on the
border structure and the letter-counts realised at those borders.
\end{proof}

For coins, total comparability on all binary strings characterises fairness;
see Theorem~\ref{thm:coin-totality-iff-fair}. Its proof in fact gives two
incomparable strings of the same length for every non-fair bias.

\subsection{Fair dice enforce a total pre-order}
Given strings $U$ and $V$, for fair dice we have (Remark~\ref{fair})
\[
h_T(x)\;=\;\sum_{\ell\in\mathcal{B}(T)} \frac{s^\ell}{x^\ell},
\qquad
\text{so that}\qquad
S_U(x)-S_V(x)=(h_U(x)-h_V(x))\,A_U(x) A_V(x) .
\]

The stochastic-order statement is due to Blom~\cite{Blom1984}.  Our
contribution here is a self-contained coefficientwise proof and the explicit
independent-race consequence in Corollary~\ref{cor:fair-mean-race}.  The
coefficientwise formulation matters because the fact that
$S_U(x)-S_V(x)$ has a fixed sign for $0<x<1$ would not, by itself, imply that
all coefficients have that sign.

\begin{proof}[Proof of Theorem~\ref{thm:fair-total-order}]
Appendix~\ref{app:fair-order-proof} gives a self-contained coefficientwise
proof.  Its key step factors a same-length tail difference as
\[
S_T(x)-S_V(x)=\bigl(S_{c(T)}(x)-S_{c(V)}(x)\bigr)K_T(x)K_V(x),
\]
where $c(T)$ is the longest proper border of $T$ (or the empty word) and
$K_T,K_V$ are probability generating functions, and hence have nonnegative
coefficients.  Strong induction on word length then proves total stochastic
comparability, including unequal lengths and equality cases (dominance and
statistical equivalence are defined in Section~\ref{subsec:stoch-dom-setup}),
which is the total pre-order asserted in the statement.
\end{proof}

As an immediate consequence of Theorem~\ref{thm:fair-total-order} we obtain an explicit description of the
dominance order in the fair case.

\begin{corollary}[Ordering by expectation]
\label{cor:order-by-expectation-compact}
Under fairness,
\(
T_A \text{ dominates } T_B \iff \mathbb{E}[W_{T_A}]>\mathbb{E}[W_{T_B}].
\)
The strings are statistically equivalent iff their means are equal.
\end{corollary}

\begin{corollary}[Mean order agrees with independent racing]
\label{cor:fair-mean-race}
For a fair $s$-die and independent streams, if
$\mathbb E[W_U]<\mathbb E[W_V]$, then
\[
\mathbb P(W_U<W_V)+\tfrac12\mathbb P(W_U=W_V)>\tfrac12.
\]
If the means are equal, the two waiting times have the same distribution and
the random-tie score is exactly $1/2$.
\end{corollary}
\begin{proof}
Theorem~\ref{thm:fair-total-order} gives $W_U\prec_{\rm st}W_V$.
The independent-race implication, including strictness, is proved in
Appendix~\ref{app:fair-order-proof}.
\end{proof}

\medskip

\subsection{Coins: total comparability forces fairness}
The fair-die theorem above raises a natural converse question in the simplest nontrivial alphabet:
\emph{for binary sources, does total comparability under stochastic dominance hold only in the fair case?}
The next theorem answers this in the affirmative.

\begin{theorem}[For coins, total stochastic dominance characterises fairness]
\label{thm:coin-totality-iff-fair}
Let $p\in(0,1)$ and consider an i.i.d.\ coin with $\Pr(\mathrm{H})=p$ and $\Pr(\mathrm{T})=q:=1-p$.

Then the strict stochastic-dominance order $\prec_{\mathrm{st}}$ defined above is total up to equivalence on the set of \emph{all} binary strings
(i.e.\ every pair is comparable up to statistical equivalence) if and only if $p=\tfrac12$.
\end{theorem}

\begin{proof}
If \(p=\tfrac12\) then the source is a fair \(2\)-die, so total comparability follows from Theorem~\ref{thm:fair-total-order}.

Conversely assume \(p\neq\tfrac12\).
By symmetry (swap \(\mathrm{H}\) and \(\mathrm{T}\)) we may assume \(p>q\).
Choose \(n\ge2\) such that \(p^n<q\); such an \(n\) exists since \(p^n\to0\) as \(n\to\infty\).
Let \(A=\mathrm{H}^n\) and \(B=\mathrm{H}^{n-1}\mathrm{T}\).

Both patterns have length \(n\), hence their first-hit times satisfy \(W_A,W_B\ge n\).
At the first possible ``checkpoint'' \(m=n\) we have
\[
\Pr(W_A\le n)=\Pr(W_A=n)=p^n
\qquad\text{and}\qquad
\Pr(W_B\le n)=\Pr(W_B=n)=p^{n-1}q.
\]
Since \(p>q\) we get \(p^n>p^{n-1}q\), and therefore
\[
\Pr(W_A>n)=1-p^n \;<\;1-p^{n-1}q=\Pr(W_B>n).
\]
Thus \(W_A\not\sdom W_B\), so \(B\not\prec_{\mathrm{st}} A\).

However, using Theorem~\ref{th1}, the expected waiting times are
\[
\E[W_A]=\sum_{j=1}^n p^{-j}=\frac{p^{-n}-1}{q}
=\frac{1-p^n}{p^n q}>\frac{1-q}{p^n q}
=\frac{1}{p^{n-1}q}=\E[W_B]
\]
since \(\mathcal{B}(B)=\{n\}\).
So \(\E[W_A]>\E[W_B]\), 
and Theorem~\ref{thm:dom-implies-mean} implies that \(W_B\not\sdom W_A\).

Therefore \(A\) and \(B\) are \emph{incomparable} under stochastic dominance (and in particular not statistically equivalent), so \(\prec_{\mathrm{st}}\) cannot be total up to equivalence when \(p\neq\tfrac12\).
\end{proof}

\section{Hadamard products and winning odds in string racing games}
\label{sec:hadamard}

We now derive head-to-head winning probabilities for \emph{independent} players using generating functions. 
For a target string $T$, let $A_T(x)=\sum_{n\ge 1} a_{T,n}x^n$ be the pgf of the first occurrence time $W_T$ (Theorem~\ref{th1}).
The central operation will be the \emph{Hadamard product} of two series,
\[
(F\odot G)(x)\;:=\;\sum_{n\ge 0} f_n g_n\,x^n\qquad\text{for }F(x)=\sum f_n x^n,\;G(x)=\sum g_n x^n,
\]
which multiplies coefficients \emph{term-by-term}. In our setting this lets us pair the first-hit probability of one player at time $n$
with the other player's survival probability at time $n$. Background on Hadamard products of rational generating functions and analytic‑combinatorics tools can be found in \cite{Bostan2006,Flajolet2009}.

\paragraph*{Setup and tie conventions}
For a target string $T$, write
\[
s_{T,n}\;:=\;\Pr(W_T>n)\;=\;1-\sum_{t=0}^{n} a_{T,t}
\]
for the tail (survival) probability, 
i.e., the probability that $T$ does not appear at or before time $n$. 
The corresponding tail generating function is
\[
S_T(x)\;:=\;\sum_{n\ge 0} s_{T,n}x^n
 \;=\;\frac{1-A_T(x)}{1-x} .
\]
For $n\ge1$, the coefficient of $x^n$ in $xS_T(x)$ is the probability that $T$ does not appear \emph{before} time $n$.

Given two independent players with targets $T_1,T_2$, 
we will use three standard tie conventions:
\[
\begin{aligned}
\text{\emph{strict:}}&\quad \Pr(T_1<T_2)=\sum_{n\ge 1} a_{T_1,n}\,s_{T_2,n},\\
\text{\emph{tie-favoured for $T_1$:}}&\quad \Pr(T_1\le T_2)=\sum_{n\ge 1} a_{T_1,n}\,s_{T_2,n-1},\\
\text{\emph{random tie-break:}}&\quad \Pr_{\mathrm{rtb}}(T_1\text{ beats }T_2)=\Pr(T_1<T_2)+\tfrac12\Pr(T_1=T_2),
\end{aligned}
\]
with $\Pr(T_1=T_2)=\sum_{n\ge 1} a_{T_1,n}a_{T_2,n}$.

\medskip
\paragraph*{Hadamard formulation}
The series whose $n$-th coefficient is the probability that $T_1$ wins \emph{at} time $n$ (under the strict convention) is
\[
W_{T_1|T_2}(x)\;=\;A_{T_1}(x)\ \odot\ S_{T_2}(x).
\]
The total win probabilities are obtained as \emph{Abelian limits}:
\[
\Pr(T_1 < T_2)\;=\;\lim_{x\uparrow 1} W_{T_1|T_2}(x),\qquad
\Pr(T_1 \le T_2)\;=\;\lim_{x\uparrow 1}\bigl(A_{T_1}\odot x\,S_{T_2}\bigr)(x),
\]
and, with random tie-break,
\[
\Pr_{\mathrm{rtb}}(T_1\text{ beats }T_2)
=\lim_{x\uparrow 1}\Bigl[W_{T_1|T_2}(x)
+\tfrac12\,\bigl(A_{T_1}\odot A_{T_2}\bigr)(x)\Bigr].
\]
This evaluation uses only monotone convergence of nonnegative series, as formalised next.

\begin{lemma}[Abelian evaluation of Hadamard win series]\label{lem:abel-hadamard}
Let $A(x)=\sum_{n\ge 1} a_n x^n$ and $B(x)=\sum_{n\ge 1} b_n x^n$ have nonnegative coefficients and radius of convergence at least $1$,
with $\sum_{n\ge 1} a_n=1$ and either $0\le b_n\le 1$ for all $n$ or, more generally, $\sum_{n\ge 1} a_n b_n<\infty$.
For $H(x):=(A\odot B)(x)=\sum_{n\ge 1} a_n b_n x^n$ we have
\[
\sum_{n\ge 1} a_n b_n<\infty
\qquad\text{and}\qquad
\lim_{x\uparrow 1} H(x)=\sum_{n\ge 1} a_n b_n .
\]
\end{lemma}

\begin{proof}
Set $c_n:=a_n b_n\ge 0$. By hypothesis $\sum_n c_n<\infty$. For $x\in[0,1)$,
$H(x)=\sum_{n\ge 1} c_n x^n$ is nondecreasing in $x$ and bounded by $\sum_n c_n$.
Given $\varepsilon>0$, choose $N$ with $\sum_{n>N} c_n<\varepsilon$ and split
$H(x)=\sum_{n=1}^N c_n x^n+\sum_{n>N} c_n x^n$; the tail is $<\varepsilon$ for all $x<1$,
and the finite sum tends to $\sum_{n=1}^N c_n$ as $x\uparrow 1$. Let $x\uparrow 1$ then $N\to\infty$.
\end{proof}

\begin{remark}[Application and convention]\label{rem:abel-notation}
In our application $A(x)=A_{T_1}(x)$ with $\sum a_{T_1,n}=1$, and the coefficients of $S_{T_2}$ are $s_{T_2,n}$, which lie in $[0,1]$. Lemma~\ref{lem:abel-hadamard} therefore yields the identities above.
Henceforth, when we write \(F(1)\) for a Hadamard series, we mean its Abelian value \(\lim_{x\uparrow 1} F(x)\)
(i.e.\ we evaluate the Hadamard series at \(x=1\) directly, with no additional \(\frac{1}{1-x}\) factor outside the Hadamard product).
\end{remark}
\subsection{Paradoxical string racing games}
We consider the \emph{parallel racing game} between two strings \(A,B\): two independent
copies of the i.i.d.\ source are generated in parallel, and \(A\) (resp.\ \(B\)) \emph{finishes}
when it first appears on its stream; ties (simultaneous finishes) are broken uniformly at random.
Write \(\mathsf{Win}(A\!>\!B)\) for the probability that \(A\) defeats \(B\) under this rule.

Form the strict preference digraph on a family of targets, with an arc
\(A\to B\) iff \(\mathsf{Win}(A\!>\!B)>\tfrac12\). A \emph{paradoxical triple} is then a directed \(3\)-cycle
\(A\to B\to C\to A\). (Tied pairs have no arc. Stochastic dominance itself is acyclic for every source.)

\begin{theorem}[Paradox under bias: existence of examples]\label{thm:biased-to-paradox}
For the binary alphabet, there exist bias parameters \(p\in(0,1)\setminus\{\tfrac12\}\) and strings
\(A,B,C\) such that \((A,B,C)\) is a paradoxical triple for the parallel racing game.
In particular, we have explicit examples with a \emph{common} bias \(p\) on all three streams.
\end{theorem}

Whether every non-uniform bias necessarily yields a paradoxical triple is the question retained in the brief outlook in
Section~\ref{sec:open-problems}.

\begin{remark}[Context and evidence]
Under a fair die, the stochastic-dominance comparison yields a transitive ordering across patterns
(in particular, no cycles occur), so fairness implies the absence of paradoxes for the parallel racing
game. For biased coins, many concrete paradoxical triples are observed (including some with a common
bias \(p\)).

\end{remark}

\subsection{Algebraic structure and endpoint asymptotics of \texorpdfstring{$g_{A,B}(p)$}{gAB(p)}}
\label{subsec:structure-g}

For a binary source the whole bias vector is described by the single
parameter $p=\Pr(\mathrm H)$.  This one-dimensional parameterisation lets us
go beyond individual examples: the Hadamard formula makes each pairwise race
an algebraic function of $p$, while $p\downarrow0$ and $p\uparrow1$ are,
respectively, the heads-rare and tails-rare regimes.  We therefore collect the
comparison in the signed advantage function $g_{A,B}$ and study its algebraic
structure, zeros and endpoint asymptotics.

Fix two words $A,B$ of lengths $L_A,L_B$ over the alphabet $\{\mathrm{H,T}\}$, and write
\[
W^{\mathrm{rtb}}_{A,B}(p):=\Pr(\text{$A$ beats $B$ with random tie-break}),\qquad
g_{A,B}(p):=W^{\mathrm{rtb}}_{A,B}(p)-\tfrac12 .
\]
For a word $T$, let $h(T)$ (resp.\ $t(T)$) be the number of $\mathrm{H}$’s (resp.\ $\mathrm{T}$’s) in $T$.
Recall the Hadamard formulation: with
\[
A_T(x;p)=\frac{1}{1+(1-x)\,h_T(x;p)},\qquad
S_T(x;p)=\frac{1-A_T(x;p)}{1-x},
\]
the head-to-head probabilities are obtained by Abelian evaluation at \(x=1\):
\[
\Pr(T_1\le T_2)=\bigl(A_{T_1}\odot x\,S_{T_2}\bigr)(1),\qquad
\Pr(T_1=T_2)=\bigl(A_{T_1}\odot A_{T_2}\bigr)(1),
\]
and, with random tie-break,
\[
W^{\mathrm{rtb}}(T_1\!>\!T_2;p)
=\bigl(A_{T_1}\odot S_{T_2}\bigr)(1)
+\frac12\,\bigl(A_{T_1}\odot A_{T_2}\bigr)(1).
\]

\begin{proposition}[Rational structure of the race advantage]
\label{prop:structure-g}
For fixed binary patterns $A,B$, the advantage $g_{A,B}(p)$ is rational
and real-analytic on $(0,1)$. It has a representation $N(p)/D(p)$ with
$N,D\in\mathbb Z[p]$, $D(p)\ne0$ on $(0,1)$, and
\[
\deg N,\ \deg D\le 2|A||B|,
\]
where the degree assertion for $N$ is needed only when $N\ne0$.
\end{proposition}
\begin{proof}
Use the product of the two prefix automata, with at most $M=|A||B|$
transient states. Each round consists of two independent tosses, so the four
transition probabilities are $p^2,p(1-p),(1-p)p,(1-p)^2$. Each entry of
$I-Q(p)$ and of the absorbing reward vector $r(p)$ therefore has degree at
most two. Assign rewards $1,0,1/2$ to an $A$ win, a $B$ win and a tie.
The score vector solves $(I-Q)f=r$.

For $0<p<1$ absorption occurs almost surely, with a geometric tail: in each
block of $|A|$ rounds the first stream has a fixed positive probability of
containing $A$ as that block. Thus $I-Q$ is invertible. Cramer's rule bounds
the degrees of the numerator and denominator of each entry of $f$ by $2M$.
Subtracting $1/2$, clearing rational coefficients and cancelling common
factors gives the stated bound for $g_{A,B}$. The denominator is nonzero
throughout $(0,1)$, which also proves real analyticity.
\end{proof}

\begin{corollary}[Finitely many crossovers unless identically tied]
\label{cor:rootcount}\label{thm:rational-finite-zeros}
Either $g_{A,B}$ vanishes identically or it has at most $2|A||B|$ zeros
in $(0,1)$. Distinct words may be identically tied: for example, HT and TH
have the same waiting-time distribution for every bias.
\end{corollary}
\begin{proof}
If the reduced numerator is nonzero, its degree bounds its number of roots.
\end{proof}

The Hadamard formulation gives the score directly in terms of the two
single-pattern laws. The product automaton supplies the degree bound above
and an alternative exact way to evaluate the same score.

\subsection*{Example 1: fair coin, \(\mathrm{HH}\) vs \(\mathrm{HT}\)}
Applying the Hadamard method to \(T_1=\mathrm{HH}\) and \(T_2=\mathrm{HT}\) with a fair coin yields
\[
\Pr(\mathrm{HH}< \mathrm{HT})=\frac{39}{121},\qquad
\Pr(\mathrm{HH}=\mathrm{HT})=\frac{17}{121},\qquad
\Pr_{\mathrm{rtb}}(\mathrm{HH}\ \text{beats}\ \mathrm{HT})=\frac{95}{242}\approx 0.39256.
\]
See Appendix~\ref{appendix:hadamard-examples} for the series manipulations and an independent
verification with the joint prefix automaton.

\subsection{Length‑2 reversal: \texorpdfstring{$\mathrm{HH}$ vs $\mathrm{HT}$}{HH vs HT}}
\label{sec:len2-reversal}

For a biased coin with \(\Pr(\mathrm{H})=p\), the same Hadamard calculation gives closed‑form rational
expressions in \(p\) for
\(\Pr(\mathrm{HH}<\mathrm{HT})\) and \(\Pr(\mathrm{HH}=\mathrm{HT})\), and hence for the random tie‑break
win probability \(\Pr_{\mathrm{rtb}}(\mathrm{HH}\ \text{beats}\ \mathrm{HT})\).
(Explicit formulas are collected in Appendix~\ref{appendix:hadamard-examples}.)
As \(p\) increases from \(1/2\) to \(1\), the random tie‑break win probability increases
continuously from \(95/242\) to \(1\) and crosses \(1/2\) at
\[
p^\star \approx 0.586648066265160.. 
\]
so \(\mathrm{HH}\) is disadvantaged for \(p<p^\star\) and advantaged for \(p>p^\star\).
\paragraph*{Crossing of means and a reversal window}
The expected waiting times cross at the golden–ratio conjugate
\(p_\varphi:=\tfrac{\sqrt{5}-1}{2}\) because
\(\mathbb{E}[W_{\mathrm{HH}}]=\tfrac{1}{p}+\tfrac{1}{p^2}\) and
\(\mathbb{E}[W_{\mathrm{HT}}]=\tfrac{1}{p(1-p)}\) are equal 
 when \(p^2+p-1=0\).

Consequently, for \(p\in(p^\star,p_\varphi)\) (numerically \(p^\star\approx0.586648\),
\(p_\varphi\approx0.618034\)) there is a genuine reversal: \(\mathrm{HH}\) wins the
head‑to‑head (with random tie‑break) with probability \(>\tfrac12\) even though
\(\mathbb{E}[W_{\mathrm{HH}}]>\mathbb{E}[W_{\mathrm{HT}}]\).
By the symmetry \(p\mapsto 1-p\) (swapping \(\mathrm{H}\leftrightarrow \mathrm{T}\)), the analogous
reversal holds for \(\mathrm{TT}\) versus \(\mathrm{TH}\) on
\(p\in(1-p_\varphi,\,1-p^\star)\) (i.e.\ \(0.381966\ldots< p < 0.4133519\ldots\)).

\subsection{Analyticity in the bias and endpoint asymptotics}
\label{subsec:analyticity-endpoints}

Fix two finite binary strings $T_1,T_2$. For $p\in(0,1)$ let
\[
W^{\mathrm{rtb}}(T_1\!>\!T_2;p)
\]
be the head-to-head win probability under \emph{random tie-break} for two independent Bernoulli$(p)$
streams (see the Hadamard formulation above), and write
\[
g_{T_1,T_2}(p)\;:=\;W^{\mathrm{rtb}}(T_1\!>\!T_2;p)\;-\;\tfrac12 .
\]
We also write $h(T)$ and $t(T)$ for the number of $\mathrm{H}$’s and $\mathrm{T}$’s in $T$ (so $|T|=h(T)+t(T)$).

\begin{theorem}[Endpoint limits and symmetry]
\label{thm:endpoint-trichotomy}
Let $T_1,T_2$ be fixed targets. Use independent Bernoulli$(p)$ streams with random tie-break.
As $p\downarrow 0$:
\[
\begin{aligned}
\text{\emph{(A)}}\;\;&h(T_1)<h(T_2) \;\Longrightarrow\;
W^{\mathrm{rtb}}(T_1\!>\!T_2;p)\to 1,\\[2pt]
\text{\emph{(B)}}\;\;&h(T_1)>h(T_2) \;\Longrightarrow\;
W^{\mathrm{rtb}}(T_1\!>\!T_2;p)\to 0,\\[2pt]
\text{\emph{(C)}}\;\;&h(T_1)=h(T_2)\ge 1 \;\Longrightarrow\;
W^{\mathrm{rtb}}(T_1\!>\!T_2;p)\to \tfrac12,\\[2pt]
\text{\emph{(D)}}\;\;&h(T_1)=h(T_2)=0\\
&\quad\Longrightarrow\;
W^{\mathrm{rtb}}(T_1\!>\!T_2;p)\to \mathbf{1}_{\{|T_1|<|T_2|\}}+\tfrac12\,\mathbf{1}_{\{|T_1|=|T_2|\}}.
\end{aligned}
\]
By complement symmetry ($\mathrm{H}\leftrightarrow \mathrm{T}$), as $p\uparrow 1$ the same statements hold with
$h(\cdot)$ replaced by $t(\cdot)$.
\end{theorem}

\begin{proof}
First consider a fixed word $T$ with $k=h(T)\ge1$ and put $r=p^k$.
Every proper border contains fewer than $k$ heads. Indeed, a border of
length $b<|T|$ makes $d=|T|-b$ a period. If that border contained all heads,
the final $d$ letters would all be T; shifting any position repeatedly by
the period would force every letter to be T, a contradiction.

For fixed $\lambda>0$, set $x=\exp(-\lambda r)$ in Theorem~\ref{th1}.
The full-border term in $r h_T(x)$ tends to one. Each proper-border term
tends to zero, since its head count is smaller than $k$. Consequently
\[
\mathbb E[e^{-\lambda p^kW_T}]
=\frac{1}{1+(1-e^{-\lambda r})h_T(e^{-\lambda r})}
\longrightarrow\frac{1}{1+\lambda}.
\]
The continuity theorem for Laplace transforms gives
$p^kW_T\Rightarrow\mathrm{Exp}(1)$.

If two words have the same positive head count, independence gives joint
convergence to two independent rate-one exponentials. Their difference has
no atom at zero, so the strict win probability tends to $1/2$ and the tie
probability tends to zero. This proves (C).

If $0<h(T_1)<h(T_2)$, choose $\beta$ strictly between the two head counts.
The convergence just proved implies
$\mathbb P(W_{T_1}>p^{-\beta})\to0$ and
$\mathbb P(W_{T_2}\le p^{-\beta})\to0$, proving (A). When $h(T_1)=0$,
$W_{T_1}$ converges in probability to $|T_1|$, whereas the other waiting
time tends to infinity in probability. This proves the remaining part of
(A); interchanging the words proves (B).

An all-T word of length $L$ satisfies $\mathbb P(W_T=L)=(1-p)^L\to1$.
This proves (D), including the tie case. Complement symmetry gives the
statements at the other endpoint.
\end{proof}

\begin{remark}[Equal-count endpoint rates]
If $h(T_1)=h(T_2)\ge1$, then $g_{T_1,T_2}(p)=O(p)$ as $p\downarrow0$:
the rational function has limit zero and hence a removable singularity at
zero. The analogous statement holds at one when the positive tail counts
agree. Head-count differences alone do not determine the error rate in
(A) and (B). For example,
\[
W^{\mathrm{rtb}}(\mathrm{HH}>\mathrm T;p)
=\frac{p^3(1+p)}{2(p^4-p^3+p^2-p+1)}
=\tfrac12p^3+O(p^4).
\]
\end{remark}

\begin{corollary}[No endpoint clustering of crossovers]
\label{cor:no-endpoint-cluster}
For any pair $(T_1,T_2)$ with $g_{T_1,T_2}\not\equiv0$, $g_{T_1,T_2}$ has only finitely many zeros in $(0,1)$,
and there exists $\varepsilon>0$ such that:
\begin{itemize}\itemsep2pt
\item if $h(T_1)\ne h(T_2)$ then $g_{T_1,T_2}(p)$ has a fixed nonzero sign on $(0,\varepsilon]$;
\item if $t(T_1)\ne t(T_2)$ then $g_{T_1,T_2}(p)$ has a fixed nonzero sign on $[1-\varepsilon,1)$.
\end{itemize}
\end{corollary}

\begin{proof}
The finiteness of zeros is Corollary~\ref{thm:rational-finite-zeros}. The fixed‑sign claims follow from the limits in Theorem~\ref{thm:endpoint-trichotomy}.
\end{proof}

\begin{remark}[A practical filter for computations]
\label{rem:practical-endpoint-filter}
Corollary~\ref{cor:no-endpoint-cluster} gives two cheap, certified screening tests (filters) for Stage~1/Stage~2 searches:
(i) if $h(T_1)\ne h(T_2)$, there can be no crossover arbitrarily near $p=0$ and the near‑$0$ orientation is determined by $h(\cdot)$; (ii) if $t(T_1)\ne t(T_2)$, likewise near $p=1$. In particular, pairs with different $h$ (resp.\ different $t$) cannot contribute root clusters close to the endpoints.
\end{remark}

\subsection*{Monotonicity in the coin bias \(p\)}
\label{subsec:Monotonicity in the coin bias p}

Fix two target strings \(T_1,T_2\). Let
\[
W^{\mathrm{rtb}}(p):=\Pr(\,T_1\text{ beats }T_2\,;\ \text{random tie-break};\ \Pr(\mathrm{H})=p\,),
\quad
W^{\mathrm{strict}}(p):=\Pr(W_{T_1}<W_{T_2}),
\]
for two independent Bernoulli\((p)\) streams. By the Hadamard–GF calculus above
and absorbing‑chain representation,
both maps are rational (hence real‑analytic) on \(p\in(0,1)\); in particular, \(W^{\mathrm{rtb}}(p)-\tfrac12\) either vanishes identically or has only finitely many zeros and the preference orientation is constant on each interval of the finite partition induced by those zeros (Corollary~\ref{thm:rational-finite-zeros}). Moreover, the signs near the endpoints are governed by the head/tail counts \(h(\cdot),t(\cdot)\) (Theorem~\ref{thm:endpoint-trichotomy}), so there is no clustering of crossovers at \(0\) or \(1\) (Corollary~\ref{cor:no-endpoint-cluster}).

\begin{proposition}[Non‑monotonicity occurs]
\label{prop:nonmono}
In general, neither \(W^{\mathrm{rtb}}(p)\) nor \(W^{\mathrm{strict}}(p)\) is monotone on \(p\in(0,1)\).
\end{proposition}

\noindent\emph{Witnesses and illustrations.}
\begin{itemize}\itemsep4pt

\item \textbf{Unequal lengths, random tie‑break.}
For \((T_1,T_2)=(\mathrm{HH},\mathrm{H})\),
\(W^{\mathrm{rtb}}(\mathrm{HH}>\mathrm{H};p)\) is non‑monotone: it has a unique interior maximum at
\[
p^\star \approx 0.422649725,\qquad
W^{\mathrm{rtb}}(\mathrm{HH}>\mathrm{H};p^\star)\approx 0.138963149,
\]
and vanishes at both extremes,
\(
\lim_{p\downarrow 0}W^{\mathrm{rtb}}= \lim_{p\uparrow 1}W^{\mathrm{rtb}}=0.
\)
Equivalently,
\(W^{\mathrm{rtb}}(\mathrm{H}>\mathrm{HH};p)=1-W^{\mathrm{rtb}}(\mathrm{HH}>\mathrm{H};p)\)
has an interior minimum and tends to \(1\) at both ends.

\item \textbf{Equal lengths, random tie‑break.}
For \(T_1=\mathrm{HHT}\), \(T_2=\mathrm{HTH}\),
\[
W^{\mathrm{rtb}}(0.40) \approx 0.5547588016,\quad
W^{\mathrm{rtb}}(0.50) \approx 0.5564733557,\quad
W^{\mathrm{rtb}}(0.60) \approx 0.5539977106,
\]
so \(W^{\mathrm{rtb}}(p)\) is not monotone.

\item \textbf{Equal lengths, strict ties; i.i.d.\ times.}
Let \(T=\mathrm{TH}\). With \textbf{two independent} copies \(W,W'\) of \(W_\mathrm{TH}\),
\[
W^{\mathrm{strict}}(T>T;p)=\Pr(W<W')=\frac{1-\Pr(W=W')}{2}.
\]
Writing \(\alpha:=p(1-p)\), a closed form for the tie probability is
\[
\Pr(W=W')=\frac{\alpha(1+\alpha)}{(2+\alpha)(1-\alpha)}
=\frac{\alpha(1+\alpha)}{2-\alpha-\alpha^2},
\]
and therefore
\[
W^{\mathrm{strict}}(\mathrm{TH}>\mathrm{TH};p)
=\frac{1-\alpha-\alpha^2}{2-\alpha-\alpha^2}
=\frac{p^4-2p^3+p-1}{p^4-2p^3+p-2}.
\]
In particular,
\[
\begin{aligned}
W^{\mathrm{strict}}(\mathrm{TH}>\mathrm{TH};0.40)&=\tfrac{439}{1064}\approx 0.41259398496,\\
W^{\mathrm{strict}}(\mathrm{TH}>\mathrm{TH};0.50)&=\tfrac{11}{27}\approx 0.40740740741,\\
W^{\mathrm{strict}}(\mathrm{TH}>\mathrm{TH};0.60)&=\tfrac{439}{1064}\approx 0.41259398496,
\end{aligned}
\]
so \(W^{\mathrm{strict}}(\mathrm{TH}>\mathrm{TH};p)\) is non‑monotone in \(p\).
(Indeed the expression depends on \(p\) only through \(\alpha=p(1-p)\), hence it is symmetric under \(p\mapsto 1-p\)
and has its minimum at \(p=\tfrac12\).)

Note also that \(\mathrm{HT}\) and \(\mathrm{TH}\) have identical first‑hit pgf’s for every \(p\)
(their only border is the full length \(2\), with weight \(p(1-p)\)),
hence \(W_{\mathrm{HT}}\mathrel{\overset{d}{=}}W_{\mathrm{TH}}\) and
\(W^{\mathrm{rtb}}(\mathrm{TH}>\mathrm{HT};p)\equiv\tfrac12\) while
\(W^{\mathrm{strict}}(\mathrm{TH}>\mathrm{HT};p)=W^{\mathrm{strict}}(\mathrm{TH}>\mathrm{TH};p)\).

\end{itemize}

\begin{remark}[Local geometry]
Each of the rational functions \(W^{\mathrm{rtb}}\) and \(W^{\mathrm{strict}}\), if nonconstant, has only finitely many critical points on \((0,1)\).
The rationality theorem above reduces any finer classification of their shape
to a finite algebraic problem for each fixed pair; we do not pursue a general
classification here.
\end{remark}

\section{Fairness, bias, and paradoxes}

Under a fair die, stochastic dominance induces a total pre‑order on patterns, so independent head‑to‑head racing is transitive. With bias this ordering can fail: win probabilities can move counter to mean waiting times, and directed cycles can occur. We begin with two brief motivators and then return to string racing with i.i.d.\ biased coins (or dice) and random tie‑break unless otherwise stated.

\begin{remark}[Two warm‑up motivators]\label{rem:warmups}
(i) \emph{Shared source (dependent races).} On a 12‑hour clock, start uniformly at random and let three players sit at \(12,4,8\) o’clock; the pointer then advances at unit speed. A direct check gives \(12\succ 4\), \(4\succ 8\), \(8\succ 12\) each with probability \(2/3\); if they race simultaneously, each wins with probability \(1/3\). This shows cycles can arise when players share the same randomness.

(ii) \emph{Independent waiting times.} Let three players have independent, discrete waiting times supported on \(\{1,5,9\}\), \(\{2,6,7\}\), and \(\{3,4,8\}\) respectively, each uniformly distributed. Then player 1 beats 2 with probability \(5/9\), 2 beats 3 with \(5/9\), and 3 beats 1 with \(5/9\). Thus cycles also occur with independent times.  If all three compete simultaneously, the win probabilities are \(11/27,8/27,8/27\). 
\end{remark}

These motivators are deliberately stylised; our focus is on string racing with i.i.d.\ biased coins (or dice) and random tie‑break unless otherwise stated. In this model, we already see that win probabilities can be non‑monotone in the bias \(p=\Pr(\mathrm{H})\) and can move counter to expected waiting times; see the monotonicity subsection above for examples at equal lengths.

\subsection{Bias breaks the order: reversals and cycles}

We adopt a random tie break unless stated otherwise; head-to-head odds are evaluated using the Hadamard–GF formulation from Section~\ref{sec:hadamard} (Lemma~\ref{lem:abel-hadamard}, Remark~\ref{rem:abel-notation}).
Concretely, the relevant Hadamard products of single-pattern series are formed and then evaluated at $x=1$ in the Abelian sense $\lim_{x\uparrow 1}$.
Analogous moment/pgf formulas are available for two-state Markov sources via the gambling-teams approach \cite{PozdnyakovEtAl2006}.

\subsection{The analytic converse near fairness}
\label{sec:analytic-converse}

We now prove the biased direction of Theorem~\ref{thm:coin-mean-race-iff}.
The analytic construction below covers $1/2-3/4096\le p<1/2$ uniformly.
Appendix~\ref{app:biased-converse-proof} supplies the finite certificate for
the rest of the lower half; complementation then handles $p>1/2$.

Put $q=1-p$.  For independent integer-valued waiting times define the
random-tie score
\[
\Phi(X,Y)=\mathbb P(X<Y)+\tfrac12\mathbb P(X=Y).
\]
Let $X_{m,c}$ denote the waiting time for an unbordered word of length $m$ and one-block probability $c$.
Theorem~\ref{th1} gives
\begin{equation}
A_{m,c}(x)=\frac{cx^m}{1-x+cx^m},\qquad
S_{m,c}(x)=\frac1{1-x+cx^m},\qquad
\mathbb E[X_{m,c}]=\frac1c.
\label{eq:unbordered-law-converse}
\end{equation}
For the run $H^n$, writing $\gamma=p^n$ and $\beta=q\gamma$ gives
\begin{equation}
A_{H^n}(x)=\frac{\gamma x^n(1-px)}{1-x+\beta x^{n+1}},
\qquad
\mathbb E[W_{H^n}]=\frac{1-\gamma}{\beta}.
\label{eq:run-law-converse}
\end{equation}

\paragraph{An analytic family}

Assume first that $p<1/2$.  Fix $n\ge2$ and $1\le k<n$, put $m=n+1$, and
consider
\begin{equation}
U=H^{n-k}T^{k+1},\qquad V=H^n,\qquad Z=H^nT.
\label{eq:converse-family}
\end{equation}
The words $U,Z$ are unbordered. The run $V$ can finish one toss earlier.
Our construction balances that early-finishing contribution against a
slightly larger block probability for $U$. The estimates below place the
relative probability increase between roughly $\gamma$ and $2\gamma$.
Set
\[
\alpha=p^{n-k}q^{k+1},\qquad
\beta=qp^n,\qquad
\frac\alpha\beta=\left(\frac qp\right)^k=1+\eta.
\]
The run $V=H^n$ is bordered. Using $S_V(x)=(1-A_V(x))/(1-x)$ and \eqref{eq:run-law-converse}, we obtain
\[
S_V(x)=\frac{1-\gamma x^n}{1-x+\beta x^{n+1}}=(1-\gamma x^n)S_{m,\beta}(x).
\]
Equivalently, as an exact coefficient identity for probability measures,
\begin{equation}
\mathcal L(W_V)=\mathcal L(X_{m,\beta})+\gamma\delta_n
-\gamma\mathcal L(n+X_{m,\beta}).
\label{eq:signed-run-decomposition}
\end{equation}
Here $\mathcal L(X)$ denotes the distribution of $X$, and $\delta_n$ is unit mass at $n$. Although the right side is written as a signed combination, the left side is
the genuine run law.  Since $X_{m,\alpha}\ge n+1$ and $\Phi$ is linear in
each law, \eqref{eq:signed-run-decomposition} yields
\begin{equation}
\Phi(X_{m,\alpha},W_V)=F-\gamma G,\quad
F=\Phi(X_{m,\alpha},X_{m,\beta}),\quad
G=\Phi(X_{m,\alpha},n+X_{m,\beta}).
\label{eq:race-run-decomposition}
\end{equation}
The score kernel is nondecreasing in its second argument, hence $G\ge F$ and
\begin{equation}
\Phi(X_{m,\alpha},W_V)\le(1-\gamma)F.
\label{eq:race-run-upper}
\end{equation}

\begin{lemma}[Hazard-ratio race bound]
\label{lem:hazard-race-converse}
Let $0<b\le a$, $m\ge2$, and $ma\le1/16$.  For the unbordered laws in
\eqref{eq:unbordered-law-converse}, put
\[
\rho=\frac ab\exp\!\left(\frac52m(a-b)\right).
\]
Then
\[
\Phi(X_{m,a},X_{m,b})
\le\frac12+\frac{\rho-1}{4(1-b)}.
\]
\end{lemma}

\begin{proof}
Let $s_c(t)=\mathbb P(X_{m,c}>t)$ and
$h_c(t)=\mathbb P(X_{m,c}=t\mid X_{m,c}>t-1)$, with $h_c(t)=0$ for $t<m$.
Coefficient extraction from \eqref{eq:unbordered-law-converse} gives
\[
h_c(t)=c\prod_{j=1}^{m-1}(1-h_c(t-j))^{-1}\qquad(t\ge m).
\]
Induction, Bernoulli's inequality, and $mc\le1/16$ give
$c\le h_c(t)<2c$ for $t\ge m$. These hazards define a probability law
for every $0<c\le a$: survival tends to zero because the hazards are at
least $c$, and its coefficients satisfy the stated generating function.
Thus we may differentiate throughout the interval $[b,a]$. Put
$u_c(t)=\partial\log h_c(t)/\partial\log c$ for $t\ge m$, and
$u_c(t)=0$ for $t<m$. Logarithmic differentiation gives
\[
u_c(t)=1+\sum_{j=1}^{m-1}
\frac{h_c(t-j)}{1-h_c(t-j)}u_c(t-j).
\]
The sum of the coefficients is at most $2mc/(1-2c)\le2/15$.  A second
induction therefore gives
\[
0\le u_c(t)-1\le\frac52mc\qquad(t\ge m).
\]
Integrating
$\frac{d}{dc}\log(h_c(t)/c)=(u_c(t)-1)/c$ from $b$ to $a$ shows
\[
1\le\frac{h_a(t)}{h_b(t)}\le\rho.
\]
Conditional on both players surviving to time $t-1$, write
$x=h_a(t)$ and $y=h_b(t)$.  Given that at least one finishes at $t$, the
first player's score is
\[
\phi(x,y)=\frac{x(1-y/2)}{x+y-xy}.
\]
For $x=ry$,
\[
\phi(ry,y)-\frac12=\frac{r-1}{2(1+r-ry)}.
\]
The unconditional score is a weighted average of these conditional scores.
Since $y<2b$ and $r\ge1$, the denominator satisfies
$1+r-ry\ge2(1-b)$, proving the claim.
\end{proof}

Combining Lemma~\ref{lem:hazard-race-converse} with
\eqref{eq:race-run-upper}, the words in
\eqref{eq:converse-family} form a strict reversal whenever
\begin{equation}
\eta>\frac\gamma{1-\gamma},\qquad
m\alpha\le\frac1{16},\qquad
\rho-1<\frac{2\gamma(1-\beta)}{1-\gamma},
\label{eq:three-converse-conditions}
\end{equation}
where $\rho=(1+\eta)\exp(\frac52m(\alpha-\beta))$.
Indeed, the first condition is exactly
\[
\mathbb E[W_U]=\frac1\alpha
<\frac{1-\gamma}{\beta}=\mathbb E[W_V],
\]
and the other two make the upper bound in
\eqref{eq:race-run-upper} strictly smaller than $1/2$.

\begin{lemma}[Uniform choice near fairness]
\label{lem:near-fair-converse}
For every integer $n\ge8$, all three conditions in
\eqref{eq:three-converse-conditions} hold strictly throughout each of the
closed intervals
\begin{equation}
\begin{split}
I_{n,1}&=\left[\frac12-\frac3{2^{n+3}},\frac12-\frac5{2^{n+4}}\right],\\
I_{n,3}&=\left[\frac12-\frac5{2^{n+5}},\frac12-\frac3{2^{n+5}}\right],
\end{split}
\label{eq:k13-intervals}
\end{equation}
with $k=1$ and $k=3$, respectively. Moreover,
\begin{equation}
\bigcup_{n\ge9}\bigl(I_{n,1}\cup I_{n-1,3}\bigr)
=\left[\frac12-\frac3{4096},\frac12\right)
=[0.499267578125,1/2).
\label{eq:k13-cover}
\end{equation}
\end{lemma}

\begin{proof}
Put $\delta=1/2-p$ and $d=q/p-1=4\delta/(1-2\delta)$.
On $I_{n,1}$ we have
$5/2^{n+4}\le\delta\le3/2^{n+3}$, while on $I_{n,3}$ we have
$3/2^{n+5}\le\delta\le5/2^{n+5}$.
In either case, $\delta\le3/(8\,2^n)$; hence, for $n\ge8$,
\begin{equation}
\gamma=p^n<\frac1{256},\qquad
m\gamma<\frac9{256},\qquad q<\frac{251}{500}.
\label{eq:k13-basic-bounds}
\end{equation}
Here we used that $(n+1)/2^n$ decreases with $n$.

Set
\[
r=\frac{kd}{\gamma}
=\frac{4k\delta\,2^n}{(1-2\delta)^{n+1}}.
\]
For $k=1$, the numerator lies in $[5/4,3/2]$ and
$2(n+1)\delta\le27/1024$. For $k=3$, the numerator lies in
$[9/8,15/8]$ and the sharper bound $\delta\le5/(32\,2^n)$ gives
$2(n+1)\delta\le45/4096$. Bernoulli's inequality therefore gives
\[
\begin{array}{ll}
k=1:&\displaystyle
\frac54\le r\le\frac{3/2}{1-27/1024}<\frac{19}{10},\\[6pt]
k=3:&\displaystyle
\frac98\le r\le\frac{15/8}{1-45/4096}<\frac{19}{10}.
\end{array}
\]
For $k=1$, $\eta=d=r\gamma$. For $k=3$, we have
$d\le10/4091<1/400$ and
\[
\eta=3d+3d^2+d^3
=r\gamma(1+d+d^2/3).
\]
Since
$(19/10)(1+1/400+1/480000)<191/100$, both cases satisfy
\begin{equation}
\frac98\gamma\le\eta<\frac{191}{100}\gamma.
\label{eq:k13-eta-bounds}
\end{equation}

The first condition follows from
$1/(1-\gamma)<256/255<9/8$. For the second,
\[
m\alpha=mq\gamma(1+\eta)
<\frac9{256}\frac{251}{500}
 \left(1+\frac{191}{25600}\right)<\frac1{16}.
\]
For the third, put $z=(5/2)m(\alpha-\beta)=(5/2)m\beta\eta$.
Then
\[
\frac z\gamma
<\frac52\frac9{256}\frac{251}{500}\frac{191}{100}
<\frac{17}{200},\qquad z<\frac{17}{51200}<1.
\]
Using $e^z\le1/(1-z)$, we obtain
\[
\begin{split}
\rho-1&=(1+\eta)e^z-1\le\frac{\eta+z}{1-z}\\
&<\frac{191/100+17/200}{1-17/51200}\,\gamma
=\frac{102144}{51183}\,\gamma
<2\gamma<\frac{2\gamma(1-\beta)}{1-\gamma}.
\end{split}
\]
The last inequality uses $\beta=q\gamma<\gamma$. These estimates hold
throughout the closed intervals, including their endpoints.

Finally put $a_n=3/2^{n+3}$ and $b_n=5/2^{n+4}$.
In the coordinate $\delta=1/2-p$, the intervals $I_{n,1}$ and
$I_{n-1,3}$ are $[b_n,a_n]$ and $[a_{n+1},b_n]$, respectively.
Their union is $[a_{n+1},a_n]$. Since $a_{n+1}=a_n/2$ and $a_n\to0$,
the union over $n\ge9$ is $(0,a_9]$ in this coordinate, proving
\eqref{eq:k13-cover}. All pairs used have indices at least $8$, and every
joining point satisfies the three conditions strictly.
\end{proof}

The alternating cover is illustrated in Fig.~\ref{fig:k13-cover}.
For the plots, write the margins in
\eqref{eq:three-converse-conditions} as
\[
z_0=\eta-\frac{\gamma}{1-\gamma},\qquad
z_1=\frac1{16}-m\alpha,\qquad
z_2=\frac{2\gamma(1-\beta)}{1-\gamma}-(\rho-1),
\]
and put $M_{n,k}(p)=\min\{z_0,z_1,z_2\}$.
A positive margin certifies a strict reversal for
$H^{n-k}T^{k+1}$ versus $H^n$. The lemma proves the cover; the plots
show how the two families alternate.

\begin{figure}[p]
\centering
\begin{tikzpicture}
\node[anchor=south west,inner sep=0] (plot) at (0,0)
{\includegraphics[width=0.91\linewidth]{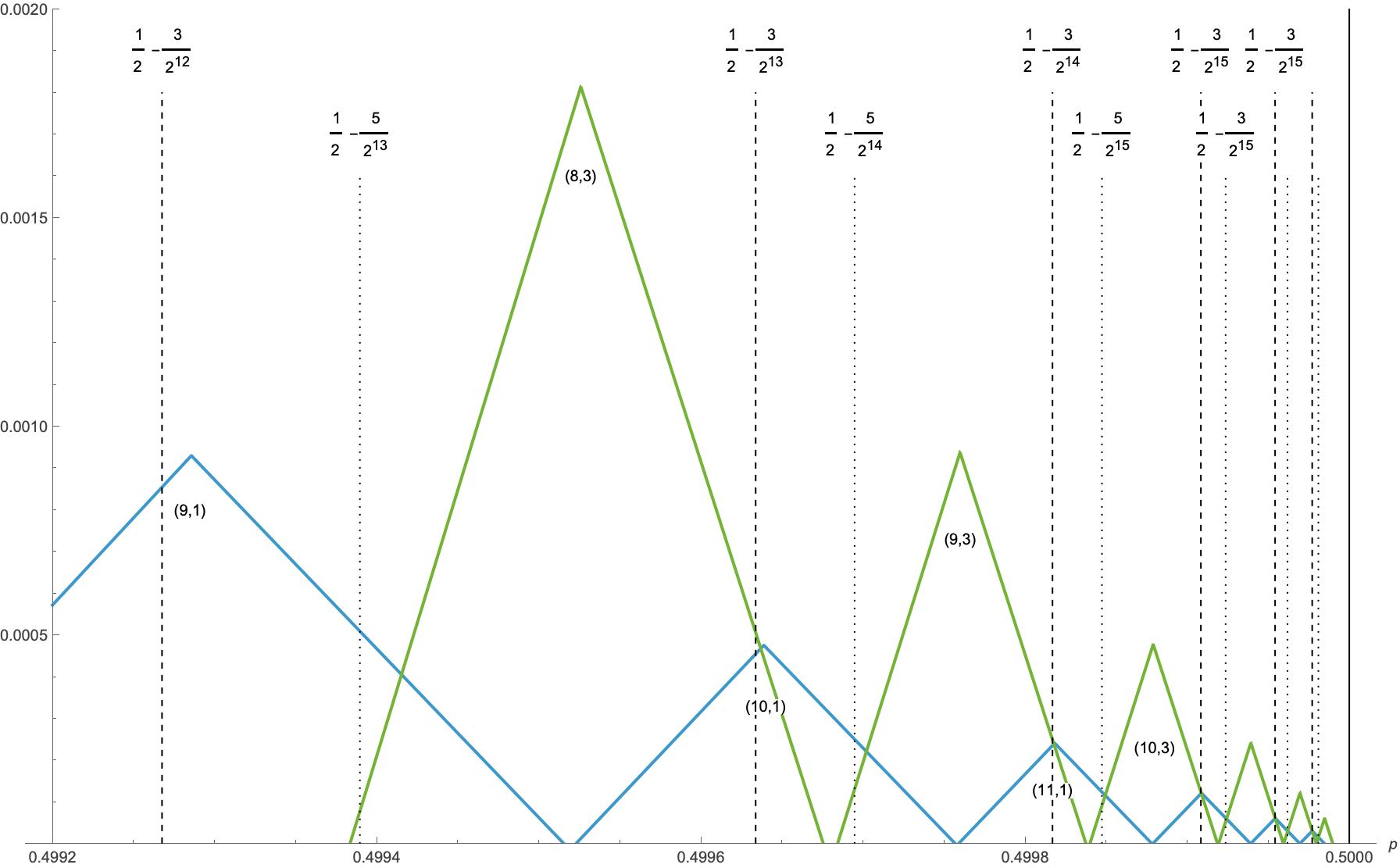}};
\begin{scope}[x={(plot.south east)},y={(plot.north west)}]
\fill[white] (0.886,0.910) rectangle (0.937,0.977);
\node[font=\scriptsize,inner sep=0] at (0.911,0.944)
{$\frac12-\frac3{2^{16}}$};
\fill[white] (0.850,0.814) rectangle (0.902,0.875);
\node[font=\scriptsize,inner sep=0] at (0.876,0.844)
{$\frac12-\frac5{2^{16}}$};
\end{scope}
\end{tikzpicture}\\[3pt]
{\small (a) Linear coordinates: $p$ and $M_{n,k}(p)$.}\\[9pt]
\begin{tikzpicture}
\node[anchor=south west,inner sep=0] (plot) at (0,0)
{\includegraphics[width=0.91\linewidth]{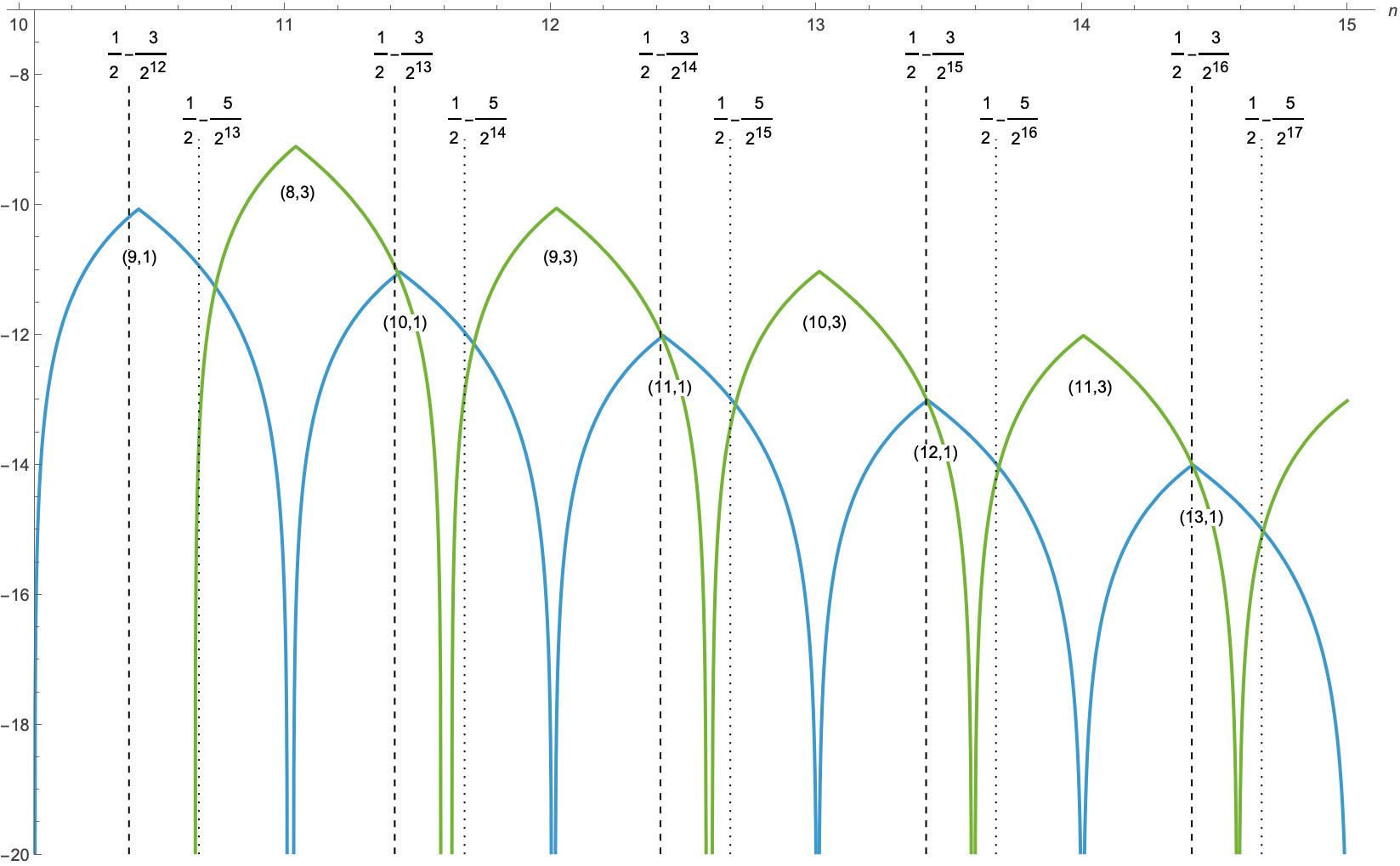}};
\begin{scope}[x={(plot.south east)},y={(plot.north west)}]
\fill[white] (0.985,0.967) rectangle (1,1);
\node[font=\scriptsize,inner sep=0,anchor=north east] at (0.999,0.997) {$x$};
\end{scope}
\end{tikzpicture}\\[3pt]
{\small (b) Logarithmic coordinates: $-\log_2(1/2-p)$ and
$\log_2 M_{n,k}(p)$.}
\caption{The $k=1$ family (blue) and $k=3$ family (green) near fairness.
The curves are labelled by $(n,k)$, and only positive margins are shown.
The vertical dashed and dotted lines mark endpoints of the certified
intervals in \eqref{eq:k13-intervals}; the labels give their values of $p$.
In panel (b), $x=-\log_2(1/2-p)$ is a real coordinate, not the integer
index of a witness.
Negative vertical values in that panel are logarithms of positive margins.
Zeros of these sufficient-condition margins need not be endpoints of the
actual reversal windows.}
\label{fig:k13-cover}
\end{figure}

\begin{corollary}[Simple run-based witnesses]
\label{cor:simple-run-witnesses}
Fix $c\in(1/4,1/2)$ and set $p_n=1/2-c2^{-n}/n$.
For all sufficiently large $n$, the choice $k=n-1$ gives a strict reversal
with the particularly simple words
\[
U=HT^n,\qquad V=H^n,\qquad Z=H^nT.
\]
\end{corollary}
\begin{proof}
Write $q_n=1-p_n$, $\gamma_n=p_n^n$ and
$\eta_n=(q_n/p_n)^{n-1}-1$. Then
\[
\gamma_n\sim2^{-n},\qquad \eta_n/\gamma_n\longrightarrow4c\in(1,2),
\qquad (n+1)\gamma_n\longrightarrow0.
\]
Indeed $q_n/p_n-1\sim4c2^{-n}/n$ and
$(1-2c2^{-n}/n)^n\to1$. With $\beta_n=q_n\gamma_n$ and
$\alpha_n=\beta_n(1+\eta_n)$ we have $(n+1)\alpha_n\to0$ and
\[
\rho_n-1=\eta_n+o(\gamma_n),\qquad
\frac{2\gamma_n(1-\beta_n)}{1-\gamma_n}\sim2\gamma_n.
\]
The three conditions in \eqref{eq:three-converse-conditions} follow.
\end{proof}

These simple witnesses give open reversal windows arbitrarily close to
fairness, by continuity. This assertion is different from a cover of every
bias in a punctured neighbourhood; the uniform cover here is supplied by
Lemma~\ref{lem:near-fair-converse}. The special case $n=6,k=5$ also certifies
the whole interval $[4989/10000,9984/20000]$, extending the finite table row
in Appendix~\ref{app:biased-converse-proof}.

\begin{proposition}[Finite cover away from fairness]
\label{prop:finite-converse-cover}
For every $0<p\le4999/10000$, some finite binary words $U,V$ satisfy
$\mathbb E[W_U]<\mathbb E[W_V]$ and $\Phi(W_U,W_V)<1/2$.
\end{proposition}
\begin{proof}
The rational interval certificates and their verification are given in
Appendix~\ref{app:biased-converse-proof}.
\end{proof}

\begin{proof}[Proof of Theorem~\ref{thm:coin-mean-race-iff}]
For $p=1/2$ the assertion follows from Corollary~\ref{cor:fair-mean-race}.
For $p<1/2$, Proposition~\ref{prop:finite-converse-cover} and
Lemma~\ref{lem:near-fair-converse} cover the whole lower half, with an overlap
on $[1/2-3/4096,4999/10000]$. For $p>1/2$, complement both witnesses for bias
$1-p$. Complementation preserves the waiting-time laws, independence and
ties, giving the required strict reversal.
\end{proof}

\subsection{Further reversals and cycles}

For targets $X,Y$ on independent copies of the source, write
\[
W_{\mathrm{rtb}}(X>Y)
:=\Pp(W_X<W_Y)+\tfrac12\Pp(W_X=W_Y)
\]
for the random tie-break score of $X$ against $Y$.

\paragraph{Embedding coin paradoxes into $s$-dice (an ``almost coin'' construction)}
A biased coin is the $s=2$ special case of an $s$-die. While one can obtain an i.i.d.\ Bernoulli stream from an $s$-die by discarding all outcomes outside a chosen pair of faces, this \emph{filtering} does not preserve waiting times for fixed patterns. However, there is a simple way to fix that issue by making the extra faces \emph{rare}, so that the $s$-die behaves almost like a coin.
(If some \(p_a=0\) then any pattern containing \(a\) has \(W_T=\infty\) almost surely, and the discussion reduces to the smaller alphabet of positive-probability symbols.)

\begin{proposition}[Coin reversals persist for $s$-dice with rare extra faces]
\label{prop:coin-to-sdie}
Fix $s\ge 3$.  Choose two faces (denoted $\mathrm{H},\mathrm{T}$) and let $p\in(0,1)$.
For $\varepsilon\in(0,1)$ define an $s$-die by
\[
\Pp(\mathrm{H})=(1-\varepsilon)p,\qquad
\Pp(\mathrm{T})=(1-\varepsilon)(1-p),\qquad
\Pp(\text{each other face})=\frac{\varepsilon}{s-2}.
\]
Let $T_1,T_2$ be two targets over $\{\mathrm{H},\mathrm{T}\}$, raced on
independent copies of the source.
If at $\varepsilon=0$ (i.e.\ for the biased coin) we have a \emph{strict} reversal
\[
\E[W_{T_1}]>\E[W_{T_2}]
\qquad\text{and}\qquad
W_{\mathrm{rtb}}(T_1>T_2)>\tfrac12,
\]
then there exists $\varepsilon_0>0$ such that the same strict reversal holds for all
$\varepsilon\in(0,\varepsilon_0)$ for the above $s$-die.
\end{proposition}

\begin{proof}[Proof]
For fixed targets, both $\E[W_T]$ and $W_{\mathrm{rtb}}(T_1>T_2)$ can be expressed by solving a finite linear system for an absorbing Markov chain on prefix states; equivalently, they are rational functions of the face probabilities.
In particular, they vary continuously with the probability vector.
Since the reversal inequalities are strict at $\varepsilon=0$, they remain valid for all sufficiently small $\varepsilon>0$.
\end{proof}

Proposition~\ref{prop:coin-to-sdie} shows that \emph{every} biased-coin reversal window yields (for each $s\ge 3$) a corresponding open set of biased $s$-dice exhibiting a reversal among targets supported on two faces.
This perturbation result supplies open sets of full-support examples for every
$s\ge3$.  We do not attempt a classification of all non-uniform $s$-dice in
the present paper.

\paragraph{Non-transitive cycles for a single biased $3$-die}
Reversals already occur for binary alphabets; the next phenomenon is non-transitivity.
We investigated cycles for a single biased $3$-die with face probabilities
$(p_0,p_1,p_2)\in\Delta^\circ:=\{p_i>0,\ p_0+p_1+p_2=1\}$, where three players race on independent copies of the same source.

\begin{proposition}
\label{prop:s3L2-no-cycle}
For $s=3$ and targets of length $L=2$, there is no triple $(T_1,T_2,T_3)$ and  $(p_0,p_1,p_2)\in\Delta^\circ$ such that
\[
W_{\mathrm{rtb}}(T_1>T_2)>\tfrac12,\qquad
W_{\mathrm{rtb}}(T_2>T_3)>\tfrac12,\qquad
W_{\mathrm{rtb}}(T_3>T_1)>\tfrac12 .
\]
\end{proposition}
\begin{proof}
This is a special case of the analytic length-two argument in
Theorem~\ref{fact:L2-no-cycles}.
\end{proof}

\begin{proposition}
\label{prop:s3L3-cycle}
For $s=3$ and $L=3$, there exists a non-empty open region $\mathcal{R}\subset\Delta^\circ$ on which the length-$3$ targets
\[
A=000,\qquad B=020,\qquad C=001
\]
form a strict directed cycle
\[
W_{\mathrm{rtb}}(A>B)>\tfrac12,\qquad
W_{\mathrm{rtb}}(B>C)>\tfrac12,\qquad
W_{\mathrm{rtb}}(C>A)>\tfrac12,
\]
i.e.\ $000>020>001>000$.
In particular, the cycle holds at
\[
(p_0,p_1,p_2)=\Bigl(\frac{624}{1468},\ \frac{399}{1468},\ \frac{445}{1468}\Bigr).
\]
\end{proposition}

\begin{figure}[t]
  \centering
  \includegraphics[width=0.6\linewidth]{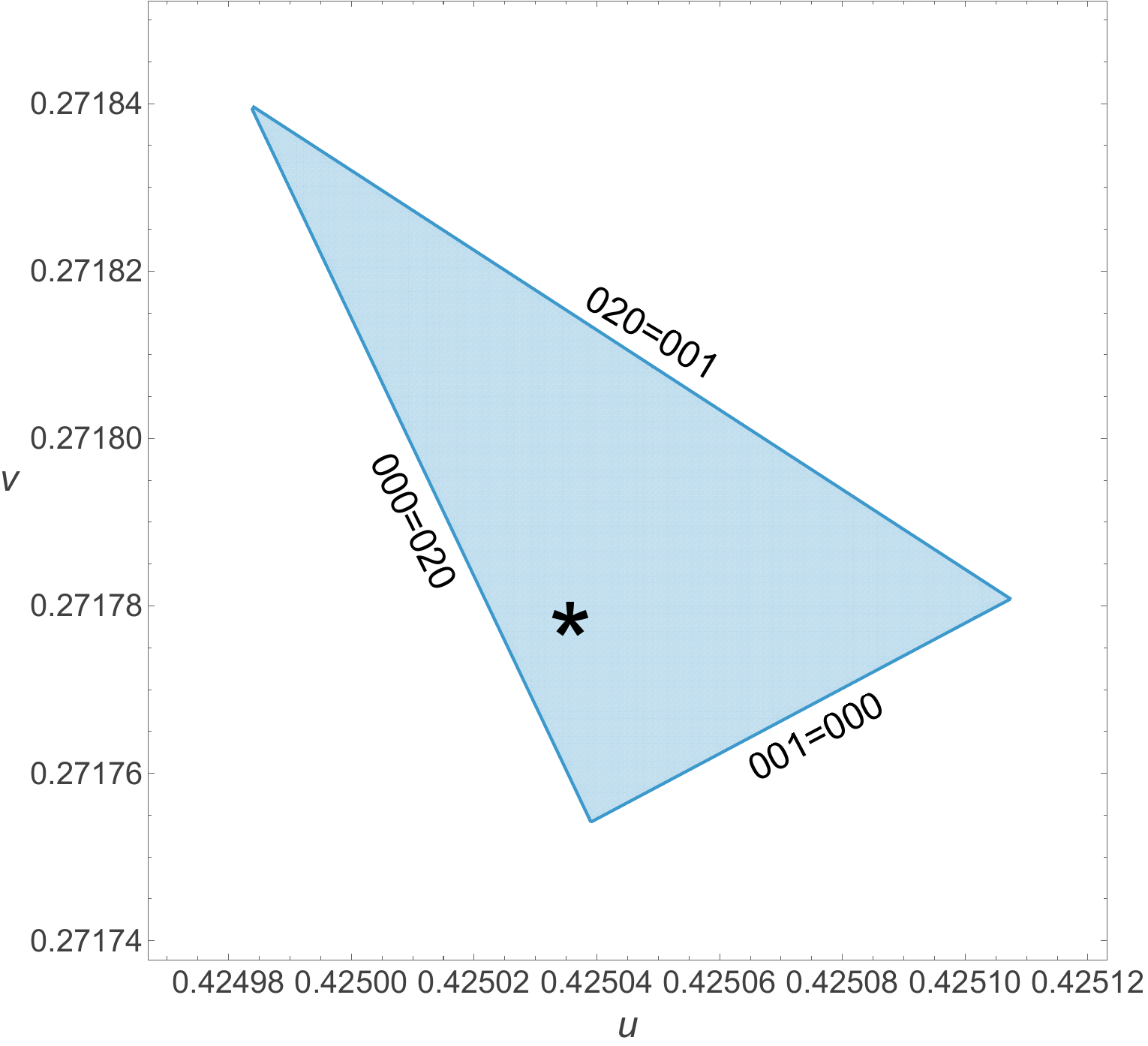}
  \caption{The non-transitivity region $\mathcal{R}\subset\Delta^\circ$ (axes $u=p_0$, $v=p_1$, so $p_2=1-u-v$) for the cycle $000>020>001>000$ under random tie-break. The boundary curves are the loci where the relevant pairwise win probabilities equal $1/2$. The asterisk indicates the point 
$(p_0,p_1,p_2)=\Bigl(\frac{624}{1468},\ \frac{399}{1468},\ \frac{445}{1468}\Bigr)$
mentioned above.}
  \label{fig:nontransitive-region}
\end{figure}

The non-transitive region has a semi-algebraic description.
Set $u:=p_0$, $v:=p_1$, and $p_2=1-u-v$, so $(u,v)$ ranges over the open simplex
\[
\{(u,v)\in \mathbb{R}^2:\ u>0,\ v>0,\ u+v<1\}.
\]
For each ordered pair $(X,Y)$ among $\{A,B,C\}$, the win advantage
$W_{\mathrm{rtb}}(X>Y)-\tfrac12$ is a rational function of $(u,v)$ with a denominator that is strictly positive on the open simplex (one may use the determinant of $I-Q$, with $Q$ the transient matrix of the absorbing chain).
Hence each strict inequality $W_{\mathrm{rtb}}(X>Y)>\tfrac12$ is equivalent to a strict polynomial sign condition on the corresponding numerator.
For the cycle $A=000$, $B=020$, $C=001$, the region is therefore the basic open semi-algebraic set
\begin{equation}\label{eq:R-cycle-def}
\mathcal{R}
=\Bigl\{(u,v)\in\mathbb{R}^2:\ u>0,\ v>0,\ u+v<1,\ 
P_{AB}(u,v)>0,\ 
P_{BC}(u,v)>0,\ 
P_{CA}(u,v)>0
\Bigr\},
\end{equation}

where
\begin{align}
P_{AB}(u,v)
&=
-u^{13}-3u^{12}v-3u^{11}v^{2}-u^{10}v^{3}
+6u^{12}+15u^{11}v+12u^{10}v^{2}+3u^{9}v^{3}\nonumber\\
&\quad
-14u^{11}-27u^{10}v-15u^{9}v^{2}-2u^{8}v^{3}
+15u^{10}+18u^{9}v+3u^{8}v^{2}-u^{7}v^{3}\nonumber\\
&\quad
-6u^{9}+u^{8}v+5u^{7}v^{2}+u^{6}v^{3}
+u^{8}-u^{7}v-5u^{7}-7u^{6}v-2u^{5}v^{2}\nonumber\\
&\quad
+4u^{6}+u^{5}v-u^{4}v^{2}+2u^{5}+4u^{4}v+u^{3}v^{2}
-3u^{4}-2u^{3}v+u^{3}+u^{2}+uv+u+v-1,
\label{eq:P000020}
\\[0.5em]
P_{BC}(u,v)
&=
u^{12}v^{3}+3u^{11}v^{4}+3u^{10}v^{5}+u^{9}v^{6}
-5u^{11}v^{3}-12u^{10}v^{4}-9u^{9}v^{5}-2u^{8}v^{6}\nonumber\\
&\quad
+10u^{10}v^{3}+18u^{9}v^{4}+9u^{8}v^{5}+u^{7}v^{6}
-10u^{9}v^{3}-12u^{8}v^{4}-3u^{7}v^{5}
-u^{9}v^{2}+2u^{8}v^{3}-u^{6}v^{5}\nonumber\\
&\quad
+3u^{8}v^{2}+5u^{7}v^{3}+3u^{6}v^{4}+3u^{6}v^{3}+3u^{5}v^{4}
-7u^{6}v^{2}-10u^{5}v^{3}-2u^{4}v^{4}
+7u^{5}v^{2}+4u^{4}v^{3}\nonumber\\
&\quad
+2u^{5}v+2u^{4}v^{2}+2u^{3}v^{3}
-4u^{4}v-4u^{3}v^{2}+2u^{3}v
+u^{2}v+uv^{2}-uv-u-2v+1,
\label{eq:P020001}
\\[0.5em]
P_{CA}(u,v)
&=
u^{12}v^{3}-2u^{11}v^{3}+u^{10}v^{3}
+u^{8}v^{2}-2u^{7}v^{2}-u^{6}v^{2}+2u^{5}v^{2}\nonumber\\
&\quad
-u^{5}v-u^{4}v+u^{2}v+uv-u+v.
\label{eq:P001000}
\end{align}

Equations \eqref{eq:R-cycle-def}--\eqref{eq:P001000} provide a complete algebraic description of the non-transitivity region for the cycle $000>020>001>000$ under random tie-break (see Fig.~\ref{fig:nontransitive-region}). The region is roughly triangular, but is not a triangle since the polynomials do not have linear factors.

\paragraph{Computational evidence (coins, i.e., 2-dice, patterns of length $\le 8$)}
We performed an exhaustive, certified computation over all ordered pairs of binary patterns $(T_1,T_2)$ with $|T_1|,|T_2|\le 8$ (under the same tie convention as in the main text), and for each pair determined the maximal bias interval(s) on which
\[
\mathbb{E}[W_{T_1}(p)]<\mathbb{E}[W_{T_2}(p)]
\quad\text{but}\quad
\mathbb{P}(T_1\ \text{wins against}\ T_2;\,p)<\tfrac12.
\]
Let $\mathcal{R}_{\le 8}$ denote the union of all such reversal windows (an open subset of $(0,1)$).
The certified $k\le8$ reversal database gives
\[
\mathcal R_{\le8}
=(0,p_0)\cup(p_0,p_-)\cup(p_+,1-p_0)\cup(1-p_0,1),
\]
where
\[
p_0\approx0.4996747,\qquad
p_-\approx0.4996837,\qquad
p_+=1-p_-\approx0.5003163.
\]
Thus, relative to $(0,1)$,
\[
(0,1)\setminus\mathcal R_{\le8}
=\{p_0,1-p_0\}\cup[p_-,p_+].
\]
The central closed no-reversal interval has width
$p_+-p_-\approx6.326\times10^{-4}$ and is centred exactly at $1/2$.
Apart from the two isolated crossover biases $p_0$ and $1-p_0$, every
$p\notin[p_-,p_+]$ admits an explicit reversal among patterns of length at
most $8$.  These endpoints are real algebraic numbers; see
Appendix~\ref{app:reversal-census-k8}, which also gives the window counts by
pattern length.
\paragraph{Non-monotonicity in the common bias $p$}
Even with identical, independent coins ($\Pr(\mathrm{H})=p$), the win probability $W(p)$ for $T_1$ vs.\ $T_2$ need not be monotone in $p$; see Section~\ref{subsec:Monotonicity in the coin bias p} for examples (including equal-length pairs). We do not repeat these details here.

\paragraph{From reversals to cycles (independent biased coins)}
Bias also enables non-transitive $3$-cycles for independent sources.

\begin{proposition}[Length-$3$ non-transitive cycle with independent coins]
\label{prop:L3-cycle-compact}
There exist length-$3$ targets $T_A,T_B,T_C$ and biases $p_A,p_B,p_C$ such that
\[
\Pr(A\!>\!B)>\tfrac12,\quad \Pr(B\!>\!C)>\tfrac12,\quad \Pr(C\!>\!A)>\tfrac12
\]
under random tie-break, while $\mathbb{E}[W_{T_A}]>\mathbb{E}[W_{T_B}]>\mathbb{E}[W_{T_C}]$.
\end{proposition}

\noindent\textit{Witness.}
One instance is
\(
T_A=\mathrm{HHH},\ p_A=0.61;\ 
T_B=\mathrm{HTH},\ p_B=0.71;\ 
T_C=\mathrm{HHT},\ p_C=0.51,
\)
evaluated via the Hadamard--GF method (cross‑checked by the joint prefix automaton).

To illustrate the robustness of this paradox, we visualise the solution space in Figure~\ref{fig:mcdonalds}. The plot illustrates the region of bias triples $(p_A, p_B, p_C)$ for which the non-transitive cycle $A \to B \to C \to A$ holds. The resulting volume forms distinctive arch-like structures  
 within the unit cube, showing that the phenomenon persists over a continuous range of parameters rather than being an isolated singularity.

\begin{figure}[h!]
    \centering
    \includegraphics[width=0.55\linewidth]{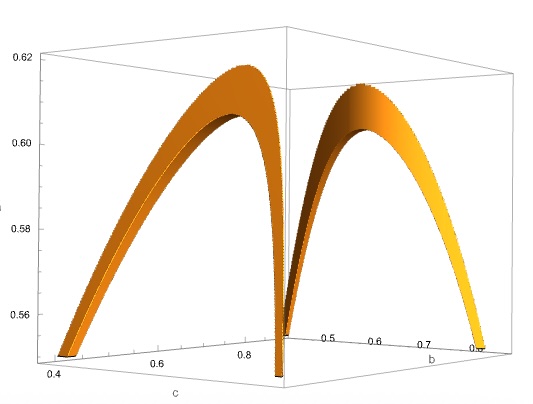} \includegraphics[width=0.40
    \linewidth]{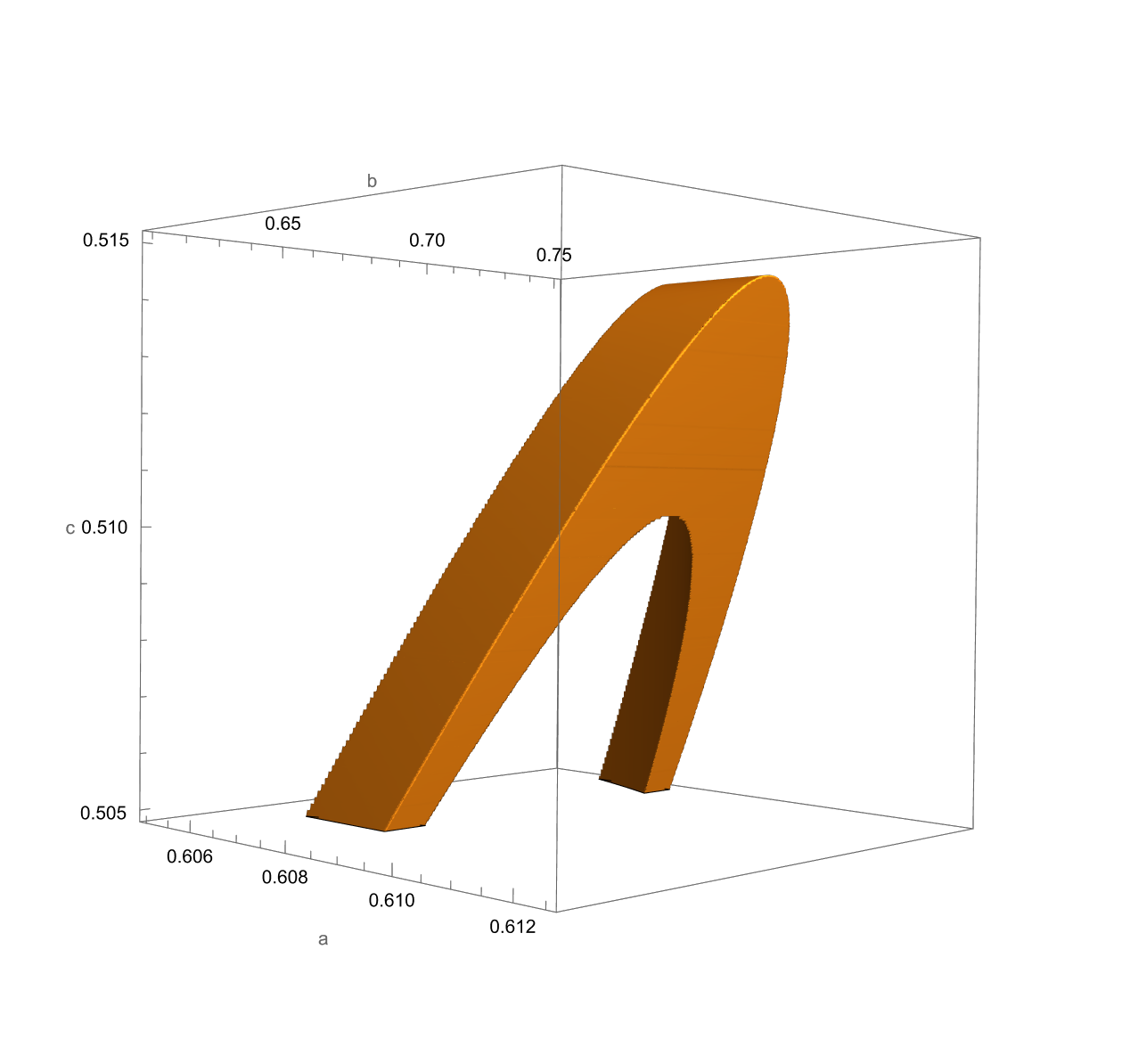} 
    \caption{Numerical illustration of the region of independent bias triples $(p_A, p_B, p_C)$ for which the patterns $T_A=\mathrm{HHH}$, $T_B=\mathrm{HTH}$, and $T_C=\mathrm{HHT}$ form a non-transitive cycle under random tie-break (axes are $p_A,p_B,p_C\in(0,1)$). On the left, the twin arches are shown, while the right shows in detail the peak of one arch. }
    \label{fig:mcdonalds}
\end{figure}

\paragraph{Common bias, unequal lengths}
As shown by the certified census in Theorems~\ref{fact:upto5-two} and~\ref{fact:upto8-counts},
non-transitive $3$-cycles already occur for a \emph{common} bias $p$ when unequal lengths are allowed.
Explicit witnesses (including examples where the mean waiting times are strictly ordered)
are given in Appendix~\ref{appendix:eq-bias-unequal-lengths} (see also Appendix~\ref{app:255}).

\section{Computational methods and certified results}\label{sec:computations}

This section concerns binary patterns and records what we verified by computer and the algorithms we used.
Unless stated otherwise, head-to-head odds are computed under \emph{random tie-break}, for independent sources with \(\Pr(H)=p\) (common bias) or with player-specific heads biases \((p_A,p_B,p_C)\) in the unequal-bias setting.

All arithmetic is \emph{exact}: crossover points are represented as algebraic numbers and inequalities are certified by root isolation and sign evaluation.

\subsection*{Short-pattern results}

\begin{theorem}[Length two, arbitrary independent sources]
\label{fact:L2-no-cycles}
Three length-two patterns cannot form a strict random-tie cycle on
independent i.i.d. finite-alphabet streams, even if the players' source
distributions differ. All target letters are assumed to have positive
probability.
\end{theorem}
\begin{proof}
There are two families of waiting-time laws. A repeated-letter target $aa$
has the law $R_r$ of the first two consecutive successes in a Bernoulli$(r)$
sequence, where $r=\Pr(a)$. Coupling the success indicators by common uniform
random variables shows that increasing $r$ makes this waiting time
stochastically smaller.

A distinct-letter target $ab$ is unbordered. Its law $U_c$ depends only on
$c=\Pr(a)\Pr(b)\le1/4$, and its tail series is
$S_c(x)=1/(1-x+cx^2)$. For $0<b<a\le1/4$,
\[
S_b(x)-S_a(x)=(a-b)x^2S_a(x)S_b(x).
\]
The coefficients on the right are nonnegative, so $U_a$ is stochastically
no larger than $U_b$.

Any triple contains two laws from the same family. Write them as $X,Y$ with
$X$ stochastically no larger than $Y$. For every independent third law $Z$,
monotonicity of the random-tie score kernel gives
$\Phi(X,Z)\ge\Phi(Y,Z)$. Also $\Phi(X,Y)\ge1/2$: compare $X$ with an
independent copy of $Y$, and use the score $1/2$ between two independent
copies of $Y$. If $X,Y$ tie in score there is no strict cycle. Otherwise the
only possible orientation is $X\to Y$, and a cycle would require
$\Phi(Y,Z)>1/2>\Phi(X,Z)$, a contradiction.
\end{proof}

This proves the length-two binary exclusion for arbitrary player biases and
the three-sided-die assertion in Proposition~\ref{prop:s3L2-no-cycle}.

\begin{theorem}[Lengths \(\le 4\), common bias]\label{fact:upto4-none}
For all triples \(T_A,T_B,T_C\) with lengths \(L_A,L_B,L_C\le 4\) and a \emph{common} bias \(p\in(0,1)\), there is \emph{no} non‑transitive triple.
\end{theorem}

\begin{theorem}[Lengths \(\le 5\), common bias]\label{fact:upto5-two}
For all triples \(T_A,T_B,T_C\) with lengths \(L_A,L_B,L_C\le 5\) and a common bias, there is exactly one non‑transitive family \emph{up to equivalence}, where we include the \(\mathrm{H}\leftrightarrow \mathrm{T}\) duality (together with \(p\mapsto 1-p\)) as part of the equivalence. Equivalently, if one distinguishes \(\mathrm{H}\) and \(\mathrm{T}\), there are exactly two dual families, each occurring on a nonempty bias interval \(I\subset(0,1)\).
For each such interval we compute explicit algebraic endpoints \(P_{\mathrm{low}}<P_{\mathrm{high}}\) so that the cycle holds for all \(p\in(P_{\mathrm{low}},P_{\mathrm{high}})\).
\end{theorem}

\begin{theorem}[Exact counts up to length \(8\)]\label{fact:upto8-counts}
For all triples with \(L_A,L_B,L_C\le 8\), we identify exactly 16 non-transitive families up to the same notion of equivalence (including \(\mathrm{H}\leftrightarrow \mathrm{T}\) together with \(p\mapsto 1-p\)).
Detailed data for these cycles, including their valid bias intervals and
representative exact algebraic crossover points, are provided in
Appendix~\ref{appendix:census} and Appendix~\ref{appendix:census-l8}.
Notably, none of these certified cycles occurs with equal lengths triples.
\end{theorem}

\subsection*{Overview of the computational pipeline}

Our computations proceed in four stages: (i) reduce the pattern space by symmetries; (ii) compute \emph{pairwise} head–to–head odds symbolically as rational functions of the bias; (iii) extract and certify all pairwise \emph{crossover points}, thus partitioning \((0,1)\) into finitely many bias intervals on which every head–to–head comparison is constant; (iv) on each such interval, form the preference digraph and detect cycles.

\paragraph{Symmetry reduction (canonical representatives)}
Let \(\mathcal{W}_L:=\{\mathrm{H},\mathrm{T}\}^{\le L}\setminus\{\epsilon\}\) be all nonempty binary words of length at most \(L\).
We reduce \(\mathcal{W}_L\) by waiting-time profile equivalence, which
preserves the \emph{single‑pattern} first-hit law for every fixed bias
\(p\in(0,1)\):
\begin{itemize}\itemsep2pt
\item By Proposition~\ref{prop:border-equivalence}, words with the same border
lengths and the same letter counts at each border length have identical
\(h_T(x;p)\), hence the same waiting-time distribution for all \(p\).  This
already identifies a word with its reversal, since reversal preserves its
border/profile data.  We keep the lexicographically smallest representative
of each profile class.

\end{itemize}
Complement \(\bar T\) is \emph{not} a symmetry at a fixed \(p\) (it corresponds instead to reflecting the bias: \((T,p)\mapsto(\bar T,1-p)\)). We use this mirror symmetry only when reporting results across \((0,1)\) to avoid listing duplicates above and below \(p=\tfrac12\).
This yields a set \(\mathcal{C}_L\subset\mathcal{W}_L\) of canonical representatives.

\paragraph{Pairwise odds as rational functions}
For each ordered pair of targets \((U,V)\), we first replace \(U\) and \(V\) \emph{individually} by their canonical representatives in \(\mathcal{C}_L\) (as defined above). We then construct the joint prefix automaton (finite absorbing Markov chain) with \(O(|U|\,|V|)\) transient states; transitions have probabilities polynomial in \(p\).
The absorption probabilities into the three absorbing states \(\{\mathbf{W}_U,\mathbf{W}_V,\mathbf{D}\}\) are entries of

\[
\mathbf{\Pi}(p)\;=\;(I-\mathbf{Q}(p))^{-1}\mathbf{R}(p),
\]
and hence are rational functions in \(p\).  With random tie‑break,
\[
W_{U>V}(p)\;:=\;\Pr(\text{$U$ beats $V$})\;=\;\Pi_{\mathbf{W}_U}(p)+\tfrac12\,\Pi_{\mathbf{D}}(p)
\] is a rational function in \(p\) with numerator/denominator in \(\mathbb{Z}[p]\).
(Equivalently, one can combine the single‑pattern pgf’s via the Hadamard/Abelian calculus from Section~\ref{sec:hadamard}; the Markov approach is convenient for symbolic elimination.)

For the exhaustive census, we evaluate the same Hadamard score using a
tail recurrence of order at most $\min\{|U|,|V|\}$. If $C$ is the companion
matrix for the tails of $V$, $\mathbf v$ their initial vector, and
$A_U(x)=\pi_U x^{|U|}/D_U(x)$, the score is the first component of
\[
\frac{\pi_U}{2}D_U(C)^{-1}
\bigl(C^{|U|}+C^{|U|-1}\bigr)\mathbf v.
\]
The supplement gives the recurrence derivation, exact determinant
implementation and complete pairwise data. Independent product-automaton
calculations provide a separate check of the evaluator.

\paragraph{Crossover set and interval decomposition}
For each ordered pair \((U,V)\) we find the numerator of
\[
\Delta_{U,V}(p)\;:=\;W_{U>V}(p)-\tfrac12\;=\;\frac{N_{U,V}(p)}{D_{U,V}(p)}.
\]
Identically zero advantages are recorded as ties and omitted from root isolation. For every other pair we compute the real roots of \(N_{U,V}\) in \((0,1)\) and isolate them by exact methods (e.g., Sturm sequences or Descartes’ rule with certified interval arithmetic).  We collect all such roots over all ordered pairs and sort them to obtain a finite partition
\[
0=p_0<p_1<\cdots<p_M< p_{M+1}=1
\]
of \((0,1)\). We call each open interval \(I_k:=(p_k,p_{k+1})\) a \emph{cell}.
On each cell \(I_k\) every \(\Delta_{U,V}\) has a constant sign (since all sign changes occur only at roots); hence the entire pairwise preference digraph is constant on \(I_k\).

\paragraph{Preference orientation on each cell}
For each cell \(I_k\) pick the midpoint \(m_k=\tfrac12(p_k+p_{k+1})\) and evaluate the signs \(\mathrm{sgn}\,\Delta_{U,V}(m_k)\) exactly.  Positive advantages define the arcs of the strict preference digraph on \(\mathcal{C}_L\) for that cell.  Non‑transitive \emph{triples} are then detected by cycle finding (see below). Whenever we report an interval of non‑transitivity \((P_{\mathrm{low}},P_{\mathrm{high}})\), we store \(P_{\mathrm{low}},P_{\mathrm{high}}\) as the exact algebraic numbers \(p_k\) and \(p_{k'+1}\) bracketing the maximal connected block of cells, including their shared boundary points, \(I_k,I_{k+1},\ldots,I_{k'}\) on which the cycle persists.

Before merging adjacent cells for a particular triple, evaluate its three
advantages at the shared algebraic endpoint. Merge only if all three remain
strictly positive there. A zero of an unrelated pair does not interrupt that
triple's interval; a relevant comparison that touches zero does, even if its
sign is positive on both sides. Apply the same rule to both inequalities
when merging reversal cells.

\subsection*{Cycle detection and paradox checks}

For a fixed cell \(I_k\), form the strict preference digraph \(\mathcal{T}_k\) on vertex set \(\mathcal{C}_L\) with arc \(U\to V\) iff \(\Delta_{U,V}(m_k)>0\).  A triple \((A,B,C)\) is non‑transitive iff
\(
A\to B,\ B\to C,\ C\to A
\)
in \(\mathcal{T}_k\).  We detect such cycles in two ways:

\begin{itemize}\itemsep2pt
\item \emph{Triple scan (baseline).} Enumerate unordered triples \(\{A,B,C\}\subset\mathcal{C}_L\) and test orientation.  This is fast for \(L\le 8\).
\item \emph{Graph sweep (faster).} For each \(A\), compute its out‑neighbours \(N^+(A)\) and its in‑neighbours \(N^-(A)\).
For each \(B\in N^+(A)\), scan \(N^+(B)\cap N^-(A)\) and record \((A,B,C)\) with \(C\in N^+(B)\cap N^-(A)\).  This avoids all \(\binom{|\mathcal{C}_L|}{3}\) checks.

\end{itemize}

For \emph{paradoxes between mean and odds}, refine the partition by adding
all roots of the nonzero mean differences
$\mathbb E[W_U]-\mathbb E[W_V]$. The signs of both comparisons are then
constant on each resulting cell. An identically zero mean difference
cannot give a strict reversal. On this refined partition, for a pair \((U,V)\) we flag cells with \(\mathbb{E}[W_U]>\mathbb{E}[W_V]\) but \(\Delta_{U,V}(m_k)>0\), using the exact expectation formulas from Theorem~\ref{th1}.

\subsection*{Optimisations and pruning}

Several cheap filters reduce the number of pairs/triples we ever examine:

\begin{itemize}\itemsep2pt
\item \emph{Border pre‑order (fair proxy).}
For common bias near \(p=\tfrac12\), the border polynomial order
(Section~\ref{sec:stoch-dom})
can suggest an order in which to examine pairs. This is a heuristic only; no case is discarded without an exact sign or dominance argument.

\item \emph{Fixed‑orientation pairs (no crossovers).}
If an ordered pair \((U,V)\) has no crossover point in \((0,1)\)
(equivalently, \(N_{U,V}(p)\) has no real root in \((0,1)\)),
then \(\Delta_{U,V}\) has a constant sign on every cell and the orientation
\(U\to V\) never changes; we therefore evaluate such pairs once and omit them
from further per‑cell updates.
In particular, if one word is a contiguous substring of the other, say \(U\subseteq V\),
then \(W_U\le W_V\) certainly (i.e.\ for every realisation of a single source), hence \(W_U\) is
stochastically no larger than \(W_V\), so under random tie‑break
\(W_{U>V}(p)\ge \tfrac12\) for all \(p\).

\item \emph{Symmetry pruning.}
Complementation identifies \((U,V,p)\) with \((\bar U,\bar V,1-p)\); it does not identify their comparisons at a fixed bias. Reversal of a word does not change
win odds, so we canonicalise pairs.

\item \emph{Dominated vertices.}
If a word \(U\) loses to (or ties with) every word of a fixed-length set on a cell,
it cannot appear in any \(3\)-cycle of that set on that cell.

\item \emph{Sparse recomputation.}
When advancing from \(I_k\) to \(I_{k+1}\), only those pairs whose polynomials had
a root at \(p_{k+1}\) can flip; we update orientations incrementally.
\end{itemize}

\subsection*{Certification details and data recorded}

For each ordered pair \((U,V)\) we store:
\begin{itemize}\itemsep2pt
\item the rational function \(W_{U>V}(p)=N_{U,V}(p)/D_{U,V}(p)\) in reduced form;
\item the ordered list of certified algebraic crossover points in \((0,1)\),
with isolating intervals;
\item the sign vector \(\bigl(\mathrm{sgn}\,\Delta_{U,V}(m_k)\bigr)_{k=0}^{M}\),
equivalently the constant sign of \(\Delta_{U,V}\) on each cell \(I_k=(p_k,p_{k+1})\).
\end{itemize}
For each non‑transitive triple we store the \emph{maximal} union of cells on which the cycle persists
and hence the exact interval(s) \((P_{\mathrm{low}},P_{\mathrm{high}})\) with algebraic endpoints.

\subsection*{Scope of the present runs}

Applying the pipeline above with symmetry reduction on \(\mathcal{W}_L\) for \(L\le 8\)  yields the certified findings listed at the beginning of this section (Theorems~\ref{fact:upto4-none}–\ref{fact:upto8-counts}).  In each reported case of non‑transitivity we provide the exact algebraic interval \((P_{\mathrm{low}},P_{\mathrm{high}})\) of common biases \(p\) for which the cycle holds; dual families appear under \(\mathrm{H}\leftrightarrow \mathrm{T}\).

\section{Brief outlook}\label{sec:open-problems}

The main characterisation concerns two-pattern reversals. The corresponding
universal existence of strict cycles remains open.

\begin{openquestion}[Cycles under a common non-fair bias]
\label{oq:paradox-iff-biased}
For a non-fair binary source, must there always be finite patterns $A,B,C$
forming a non-transitive random-tie cycle on independent streams?  If so, how
does the least possible value of
$\max\{|A|,|B|,|C|\}$ depend on the bias?  Our examples and exact short-pattern
census answer this only on particular bias intervals.
\end{openquestion}


\bibliographystyle{IEEEtran}

\bibliography{ref}

@article{Penney1969,
  author  = {Penney, Walter},
  title   = {{Problem 95: Penney's Game}},
  journal = {Journal of Recreational Mathematics},
  year    = {1969},
  volume  = {2},
  number  = {4},
  pages   = {241},
}

@article{Gardner1974,
  author  = {Gardner, Martin},
  title   = {{Mathematical Games: The paradox of the nontransitive dice and the elusive principle of indifference}},
  journal = {Scientific American},
  year    = {1974},
  volume  = {231},
  number  = {4},
  pages   = {120--125},
}

@article{Guibas1981,
  author  = {Guibas, Leonidas J. and Odlyzko, Andrew M.},
  title   = {{String overlaps, pattern matching, and nontransitive games}},
  journal = {Journal of Combinatorial Theory, Series A},
  year    = {1981},
  volume  = {30},
  number  = {2},
  pages   = {183--208},
  doi     = {10.1016/0097-3165(81)90005-4},
}

@book{Flajolet2009,
  author    = {Flajolet, Philippe and Sedgewick, Robert},
  title     = {{Analytic Combinatorics}},
  publisher = {Cambridge University Press},
  year      = {2009},
  address   = {Cambridge},
  doi       = {10.1017/CBO9780511801655},
}

@techreport{Bostan2006,
  author      = {Bostan, Alin and Flajolet, Philippe},
  title       = {{The Hadamard Product of Rational Functions}},
  institution = {INRIA},
  year        = {2006},
  number      = {RR-6002},
  url         = {https://hal.inria.fr/inria-00112873},
}

@article{Elizalde2024,
  author  = {Elizalde, Sergi and Lin, Yixin},
  title   = {{Penney's game for permutations}},
  journal = {arXiv preprint arXiv:2404.06585},
  year    = {2024},
  eprint  = {2404.06585},
  archivePrefix = {arXiv},
  primaryClass = {math.CO},
}

@article{nickerson2007penney,
  title={Penney Ante: Counterintuitive probabilities in coin tossing},
  author={Nickerson, Raymond S},
  journal={The UMAP Journal},
  volume={28},
  number={4},
  pages={503--532},
  year={2007}
}

@book{Feller1968,
  author    = {Feller, William},
  title     = {An Introduction to Probability Theory and Its Applications, Vol. I},
  edition   = {3rd},
  publisher = {Wiley},
  year      = {1968}
}

@book{Feller1971,
  author    = {Feller, William},
  title     = {An Introduction to Probability Theory and Its Applications, Vol. II},
  edition   = {2nd},
  publisher = {Wiley},
  year      = {1971}
}

@article{Li1980,
  author  = {Li, S. Y. R.},
  title   = {A martingale approach to the study of occurrence of sequence patterns in repeated experiments},
  journal = {The Annals of Probability},
  year    = {1980},
  volume  = {8},
  number  = {6},
  pages   = {1171--1176},
  doi     = {10.1214/aop/1176994578}
}

@article{GerberLi1981,
  author  = {Gerber, Hans U. and Li, S. Y. R.},
  title   = {The occurrence of sequence patterns in repeated experiments and hitting times in a {Markov} chain},
  journal = {Stochastic Processes and their Applications},
  year    = {1981},
  volume  = {11},
  number  = {1},
  pages   = {101--108},
  doi     = {10.1016/0304-4149(81)90025-9}
}

@article{FuKoutras1994,
  author  = {Fu, J. C. and Koutras, M. V.},
  title   = {Distribution Theory of Runs: A {Markov} Chain Approach},
  journal = {Journal of the American Statistical Association},
  year    = {1994},
  volume  = {89},
  number  = {427},
  pages   = {1050--1058},
  doi     = {10.1080/01621459.1994.10476841}
}

@article{Antzoulakos2001,
  author  = {Antzoulakos, D. L.},
  title   = {Waiting times for patterns in a sequence of multistate trials},
  journal = {Journal of Applied Probability},
  year    = {2001},
  volume  = {38},
  number  = {3},
  pages   = {508--518}
}

@article{Nuel2008,
  author  = {Nuel, Gr{\'e}gory},
  title   = {Pattern {Markov} Chains: Optimal {Markov} Chain Embedding Through Deterministic Finite Automata},
  journal = {Journal of Applied Probability},
  year    = {2008},
  volume  = {45},
  number  = {1},
  pages   = {226--243}
}

@article{PozdnyakovEtAl2006,
  author  = {Pozdnyakov, Vladimir and Kulldorff, Martin and Steele, J. Michael and Glaz, Joseph},
  title   = {Gambling Teams and Waiting Times for Patterns in Two-State {Markov} Chains},
  journal = {Journal of Applied Probability},
  year    = {2006},
  volume  = {43},
  number  = {1},
  pages   = {271--281},
  doi     = {10.1239/jap/1143936318}
}

@article{Blom1984,
  author  = {Blom, Gunnar},
  title   = {On the Stochastic Ordering of Waiting Times for Patterns in Sequences of Random Digits},
  journal = {Journal of Applied Probability},
  year    = {1984},
  volume  = {21},
  number  = {4},
  pages   = {730--737},
  doi     = {10.2307/3213691}
}

\appendices

\section{Direct fair-source ordering proof}
\label{app:fair-order-proof}

This appendix gives a self-contained proof of
Theorem~\ref{thm:fair-total-order} and
Corollary~\ref{cor:fair-mean-race}.  It avoids inferring coefficient signs
from the sign of a generating function on an interval.

Fix a uniform i.i.d. source on an alphabet of size $s\ge2$.

For nonnegative integer-valued random variables
write
\[
X\preceq_{\rm st}Y
\quad\Longleftrightarrow\quad
\mathbb P(X>n)\le\mathbb P(Y>n)\quad(n\ge0).
\]
Thus $Y$ is the slower variable.  For a nonempty word $T$, put
\[
A_T(x)=\mathbb E[x^{W_T}],\qquad
S_T(x)=\sum_{n\ge0}\mathbb P(W_T>n)x^n,\qquad
h_T(x)=\sum_{b\in\mathcal B(T)}s^b x^{-b}.
\]
Disjoint blocks of length $|T|$ give an exponential upper bound on the tail
of $W_T$, so the following manipulations are valid as power-series identities
and for $|x|<1$.

\begin{lemma}[Renewal identities]
\label{lem:fair-renewal-appendix}
For every nonempty word $T$,
\[
A_T(x)=\frac{1}{1+(1-x)h_T(x)},\qquad
S_T(x)=h_T(x)A_T(x),\qquad
\mathbb E[W_T]=\sum_{b\in\mathcal B(T)}s^b.
\]
For words $C,D$,
\begin{equation}
S_C(x)-S_D(x)=\bigl(h_C(x)-h_D(x)\bigr)A_C(x)A_D(x).
\label{eq:fair-tail-difference-appendix}
\end{equation}
\end{lemma}

\begin{proof}
Write $m=|T|$ and $p_j=\mathbb P(W_T=j)$.  Fix $N\ge m$ and condition the
event that a copy of $T$ ends at $N$ on the first-hit time $j$.  If
$d=N-j<m$, the two copies are compatible exactly when $m-d$ is a border
length, and the $d$ new letters have probability $s^{-d}$.  If $d\ge m$,
the terminal block lies wholly after the first copy and has probability
$s^{-m}$.  Hence
\[
s^{-m}=\sum_{j=m}^N p_jr_{N-j},\qquad
r_d=\begin{cases}
s^{-d}{\bf1}_{\{m-d\in\mathcal B(T)\}},&0\le d<m,\\
s^{-m},&d\ge m.
\end{cases}
\]
The generating function of $(r_d)$ is
\[
R(x)=s^{-m}x^m\left(h_T(x)+\frac1{1-x}\right).
\]
Multiplying the renewal equation by $x^N$, summing over $N\ge m$, and
cancelling $s^{-m}x^m$ gives the formula for $A_T$.  The elementary tail
identity $S_T=(1-A_T)/(1-x)$ then gives $S_T=h_TA_T$.  Letting $x\uparrow1$
gives the mean formula.  Finally, subtracting $h_DA_D$ from $h_CA_C$ and
using $A_T^{-1}=1+(1-x)h_T$ gives
\eqref{eq:fair-tail-difference-appendix}.
\end{proof}

For bookkeeping, adjoin the empty word $\epsilon$ with
$W_\epsilon=0$, $A_\epsilon=1$, $S_\epsilon=h_\epsilon=0$.  For nonempty
$T$, let $c(T)$ be its longest proper border, regarded as a prefix word; if
there is no nonempty proper border, set $c(T)=\epsilon$.

\begin{lemma}[Core recursion and factorisation]
\label{lem:fair-core-factor}
If $P=c(T)$ and $m=|T|$, then
\[
\mathcal B(T)=\{m\}\mathbin{\dot\cup}\mathcal B(P),\qquad
h_T(x)=s^mx^{-m}+h_P(x),\qquad
\mathbb E[W_T]=s^m+\mathbb E[W_P].
\]
Moreover, there is a nonnegative integer-valued $D_T$, independent of
$W_P$, such that
\[
W_T\stackrel d=W_P+D_T,\qquad A_T(x)=A_P(x)K_T(x),\qquad
K_T(x)=\mathbb E[x^{D_T}].
\]
Thus $K_T$ is an ordinary probability generating function with nonnegative
coefficients.  In the unbordered case $P=\epsilon$ and $K_T=A_T$.
\end{lemma}

\begin{proof}
Let $r=|P|$.  Every proper border of $T$ has length at most $r$.  Since $P$
is both a prefix and a suffix of $T$, a word of length at most $r$ is a
border of $T$ exactly when it is a border of $P$.  This proves the first
three identities.

Run the $P$- and $T$-matching automata on the same stream and stop at
$\tau=W_P$.  The suffix ending at $\tau$ cannot match a prefix of $T$ of
length $k>r$: that prefix begins with $P$, which would give an occurrence of
$P$ ending at $\tau-k+r<\tau$.  Hence the state of the $T$-automaton at
$\tau$ is exactly $P$.  Conditional on $\{\tau=n\}$, future letters are
i.i.d. uniform and independent of the first $n$ letters.  The additional
time $D_T$ from the fixed state $P$ to the first $T$ therefore has a law
independent of $W_P$, proving the factorisation.
\end{proof}

\begin{lemma}[Coefficient transfer at a fixed length]
\label{lem:fair-coefficient-transfer}
Let $T,V$ have the same length, with cores $P=c(T)$ and $Q=c(V)$.  If
$S_P-S_Q$ has nonnegative coefficients, then so does $S_T-S_V$.  Strictness
of at least one coefficient is also preserved.
\end{lemma}

\begin{proof}
Lemma~\ref{lem:fair-core-factor} and
\eqref{eq:fair-tail-difference-appendix} give the ordinary power-series
identity
\begin{equation}
S_T-S_V=(S_P-S_Q)K_TK_V.
\label{eq:fair-cauchy-transfer}
\end{equation}
The right side is a Cauchy product of series with nonnegative coefficients.
A positive coefficient cannot be cancelled because both $K_T$ and $K_V$ are
nonzero probability generating functions.
\end{proof}

\begin{proposition}[Coefficientwise stochastic order]
\label{prop:fair-coefficient-order}
For all nonempty words $U,V$,
\[
\mathbb E[W_U]<\mathbb E[W_V]
\quad\Longleftrightarrow\quad
W_U\prec_{\rm st}W_V.
\]
Equality of the means is equivalent to equality in distribution.
\end{proposition}

\begin{proof}
We prove by strong induction on $M$ that the assertion holds for all words of
length at most $M$.  At $M=1$ every waiting time is geometric with parameter
$1/s$.

The induction also includes the adjoined empty core: for every nonempty word
$Q$, $S_Q-S_\epsilon=S_Q$ has nonnegative coefficients, with strictness.

Suppose the result holds through $M-1$ and first compare length-$M$ words
$T,V$ with cores $P,Q$.  By Lemma~\ref{lem:fair-core-factor},
\[
\mathbb E[W_T]-\mathbb E[W_V]
=\mathbb E[W_P]-\mathbb E[W_Q].
\]
The induction hypothesis orders $S_P,S_Q$ coefficientwise, and
Lemma~\ref{lem:fair-coefficient-transfer} transfers that order, including
strictness, to $S_T,S_V$.

It remains to compare different lengths.  Choose distinct letters $a,b$ and
put
\[
L_M=a^{M-1}b,\qquad R_{M-1}=a^{M-1}.
\]
The word $L_M$ is unbordered and has the smallest mean $s^M$ among
length-$M$ words.  The run $R_{M-1}$ has every possible border and has the
largest mean among words of length at most $M-1$.  The same-length result and
the induction hypothesis give, for $|T|=M$ and $|V|<M$,
\[
W_V\preceq_{\rm st}W_{R_{M-1}}
\preceq_{\rm st}W_{L_M}
\preceq_{\rm st}W_T.
\]
The middle comparison is pathwise on one stream: every occurrence of $L_M$
contains its prefix $R_{M-1}$.  The chain is strict overall because, at $v=|V|<M$,
\[
\mathbb P(W_T>v)=1>1-s^{-v}=\mathbb P(W_V>v).
\]
This completes the induction.  Conversely, summing a strict tail inequality
gives a strict mean inequality.

For equality, the mean is a base-$s$ sum with digits zero or one.  If two
border sets first differ at their largest exponent $k$, then $s^k$ exceeds
$\sum_{j=1}^{k-1}s^j$.  Equal means therefore force equal border sets, and
Lemma~\ref{lem:fair-renewal-appendix} gives identical distributions.
\end{proof}

\begin{lemma}[Independent random-tie comparison]
\label{lem:fair-race-from-order}
Let independent nonnegative integer-valued $X,Y$ satisfy $X\preceq_{\rm st}Y$.  Then
\[
\mathbb P(X<Y)+\tfrac12\mathbb P(X=Y)\ge\tfrac12,
\]
strictly if the stochastic order is strict.  Equal laws give equality.
\end{lemma}

\begin{proof}
Let $F_X,F_Y$ be the distribution functions.  With one uniform random
variable $U$ on $(0,1)$, the inverse-distribution coupling
\[
\widehat X=\min\{n:F_X(n)\ge U\},\qquad
\widehat Y=\min\{n:F_Y(n)\ge U\}
\]
satisfies $\widehat X\le\widehat Y$ almost surely, and with positive
probability under strict stochastic order.  Let $Y'$ be an independent copy
of $Y$ and set
\[
g(x,y)={\bf1}_{\{x<y\}}+\tfrac12{\bf1}_{\{x=y\}}.
\]
For fixed $y$, $g(x,y)$ is nonincreasing in $x$, so
\[
\mathbb E g(\widehat X,Y')\ge
\mathbb E g(\widehat Y,Y')=\tfrac12
\]
by exchangeability of the two copies of $Y$.  In the strict case choose
$x<y$ with $\mathbb P(\widehat X=x,\widehat Y=y)>0$; the event $Y'=y$ has
positive probability and makes the last inequality strict.  Equal laws give
equality by exchangeability.
\end{proof}

Proposition~\ref{prop:fair-coefficient-order} proves
Theorem~\ref{thm:fair-total-order}; combining it with
Lemma~\ref{lem:fair-race-from-order} proves
Corollary~\ref{cor:fair-mean-race}.

\section{Exact certificate for the biased-coin converse}
\label{app:biased-converse-proof}

This appendix proves Proposition~\ref{prop:finite-converse-cover}, using the
notation and sufficient conditions of Section~\ref{sec:analytic-converse}.
All signs are certified over the rationals.

We cover $0<p\le4999/10000$.  The following exact certificate was
computed over $\mathbb Q$ by the script
\emph{biased\_converse\_certificate.py}.

For each fixed pair it computes the means from Theorem~\ref{th1} and builds
the joint prefix automaton.  A transient state records the two matched-prefix
lengths; the four next-round probabilities are $p^2,pq,qp,q^2$, and absorption
rewards are $1,0,1/2$ for a $U$ win, a $V$ win, and a tie.  Solving the finite
system $(I-Q(p))f=r(p)$ over $\mathbb Q(p)$ gives the exact random-tie score.
After denominators are cleared, signs on each rational interval are certified
by exact Bernstein coefficients: if
$P(t)=\sum_jb_jB_{j,d}(t)$ and every rational $b_j>0$, then $P>0$ on
$[0,1]$.  The script also substitutes its solutions back into the exact
linear systems.

The following fixed-word rows certify $(0,497/1000]$.
Shared endpoints are included; the endpoint $p=0$ is excluded.

\begin{center}
\begin{footnotesize}
\begin{longtable}{|c|c|c|}
\hline
\textbf{bias interval} & $U$ & $V$ \\ \hline
\endfirsthead
\hline
\textbf{bias interval} & $U$ & $V$ \\ \hline
\endhead
$(0,9/50]$ & HTTTTT & TTHTT \\ \hline
$[9/50,41/200]$ & TTT & H \\ \hline
$[41/200,23/100]$ & HTTTTTT & TTTTTTT \\ \hline
$[23/100,1/4]$ & HTTTTT & TTTTTT \\ \hline
$[1/4,7/25]$ & HTTTT & TTTTT \\ \hline
$[7/25,3147/10000]$ & HTTT & TTTT \\ \hline
$[3147/10000,67/200]$ & THTT & HH \\ \hline
$[67/200,179/500]$ & HTTT & HH \\ \hline
$[179/500,73/200]$ & TTHTTT & HHH \\ \hline
$[73/200,19/50]$ & HTTT & HHT \\ \hline
$[19/50,383/1000]$ & THTTTT & HTHH \\ \hline
$[383/1000,41/100]$ & HT & TT \\ \hline
$[41/100,11/25]$ & HTT & HH \\ \hline
$[11/25,56/125]$ & THHT & HHH \\ \hline
$[56/125,91/200]$ & HHT & THT \\ \hline
$[91/200,4721/10000]$ & HHTT & HHH \\ \hline
$[4721/10000,12/25]$ & HTTT & HHH \\ \hline
$[12/25,971/2000]$ & HHHTT & HHHH \\ \hline
$[971/2000,487/1000]$ & HHHT & THTT \\ \hline
$[487/1000,24631/50000]$ & HHHHTT & HHHHH \\ \hline
$[24631/50000,49617/100000]$ & HHHTTT & HHHHH \\ \hline
$[49617/100000,497/1000]$ & HTTTTT & HHHHH \\ \hline
\end{longtable}
\end{footnotesize}
\end{center}

For $497/1000\le p\le4999/10000$, the next rows use the analytic family
\eqref{eq:converse-family}.  The exact certificate verifies the first two
conditions of \eqref{eq:three-converse-conditions}, $1-z>0$ for
$z=\frac52m(\alpha-\beta)$, and
\[
\frac{2\gamma(1-\beta)}{1-\gamma}
-\left(\frac{1+\eta}{1-z}-1\right)>0.
\]
Since $e^z\le1/(1-z)$, this implies the final condition in
\eqref{eq:three-converse-conditions}.

\begin{center}
\begin{footnotesize}
\begin{longtable}{|c|c|c|}
\hline
\textbf{bias interval} & $n$ & $k$ \\ \hline
\endfirsthead
\hline
\textbf{bias interval} & $n$ & $k$ \\ \hline
\endhead
$[497/1000,249/500]$ & 7 & 1 \\ \hline
$[249/500,2493/5000]$ & 6 & 3 \\ \hline
$[2493/5000,4989/10000]$ & 6 & 4 \\ \hline
$[4989/10000,9984/20000]$ & 6 & 5 \\ \hline
$[9983/20000,9989/20000]$ & 7 & 4 \\ \hline
$[9989/20000,9993/20000]$ & 7 & 6 \\ \hline
$[9993/20000,24989/50000]$ & 8 & 5 \\ \hline
$[24989/50000,24993/50000]$ & 9 & 4 \\ \hline
$[24993/50000,4999/10000]$ & 9 & 6 \\ \hline
\end{longtable}
\end{footnotesize}
\end{center}

The listed intervals cover $(0,4999/10000]$ with their shared endpoints
included. The extended $n=6,k=5$ row overlaps the next row; overlap is
intentional. This proves Proposition~\ref{prop:finite-converse-cover}.


\section{Derivations for Hadamard--GF examples}
\label{appendix:hadamard-examples}

We work out in full detail the race between \(T_1=\mathrm{HH}\) and \(T_2=\mathrm{HT}\) when each player tosses an \emph{independent} fair coin at each round. We will compute:
\[
\Pr(\mathrm{HH}<\mathrm{HT}),\quad \Pr(\mathrm{HT}<\mathrm{HH}),\quad \Pr(\mathrm{HH}=\mathrm{HT}),
\]
and then the random tie-break odds.

\subsection{Generating functions for first-hit times}

Let \(A_T(x)=\sum_{n\ge 1} a_{T,n} x^n\) be the pgf of the first occurrence time of string \(T\) (as in Theorem~\ref{th1}). For a fair coin, the prefix-suffix lengths give (Theorem~\ref{th1}, Cor.~\ref{col01}).

\medskip
\paragraph{HT}
The string \(\mathrm{HT}\) has prefix-suffix lengths \(\{2\}\), hence
\[
A_{\mathrm{HT}}(x)\;=\;\frac{1}{1+(1-x)\left(\frac{2}{x}\right)^{\!2}}
\;=\;\frac{x^2}{x^2-4x+4}\;=\;\frac{x^2}{(x-2)^2}.
\]
Expanding about \(x=0\) gives
\[
A_{\mathrm{HT}}(x)\;=\;\sum_{n\ge 1}\frac{n-1}{2^n}\,x^n\quad\Longrightarrow\quad
a_{\mathrm{HT},n}=\frac{n-1}{2^n},\;\; n\ge 1.
\]
The tail \(s_{\mathrm{HT},n}:=\Pr(\mathrm{HT}\text{ not at or before time }n)=1-\sum_{t=1}^n a_{\mathrm{HT},t}\) has tail generating function 
\[
S_{\mathrm{HT}}(x)=\frac{1-A_{\mathrm{HT}}(x)}{1-x}
=\frac{1-\frac{x^2}{(x-2)^2}}{1-x}
=\frac{-4x+4}{(1-x)(x-2)^2}
=\left(1-\frac{x}{2}\right)^{-2}.
\]
Hence,
\[ s_{\mathrm{HT},n}=\frac{n+1}{2^{\,n}}\ \ (n\ge 0).
\]

\paragraph{HH}
The string \(\mathrm{HH}\) has prefix-suffix lengths \(\{1,2\}\), hence
\[
A_{\mathrm{HH}}(x)\;=\;\frac{1}{1+(1-x)\!\left(\frac{2}{x}+\frac{4}{x^2}\right)}
\;=\;\frac{x^2}{4-2x-x^2}.
\]
It is convenient to note the Fibonacci identity
\[
A_{\mathrm{HH}}(x)\;=\;\sum_{n\ge 2}\frac{F_{n-1}}{2^n}\,x^n,
\]
which follows from the classical generating function
\(\sum_{m\ge 1} F_m z^m = \frac{z}{1-z-z^2}\) upon substituting \(z=x/2\).

\subsection{Hadamard computation of head-to-head odds}

For independent sources, the series whose \(n\)-th coefficient is
\(\Pr(\mathrm{HH}\ \text{wins \emph{at} time }n)\) under tie-favoured-for-\(\mathrm{HH}\) is
\[
W_{\mathrm{HH|HT}}(x)\;=\;A_{\mathrm{HH}}(x)\ \odot\ x\,S_{\mathrm{HT}}(x),
\]
and the total tie-favoured win probability is obtained by Abelian evaluation,
\[
\Pr(\mathrm{HH}\le \mathrm{HT})=\lim_{x\uparrow 1} W_{\mathrm{HH|HT}}(x).
\]

Using the explicit coefficients, we obtain directly
\[
\Pr(\mathrm{HH}<\mathrm{HT})
=\sum_{n\ge 1} a_{\mathrm{HH},n}\,s_{\mathrm{HT},n}
=\sum_{n\ge 2}\frac{F_{n-1}}{2^n}\cdot\frac{n+1}{2^{\,n}}
=\sum_{n\ge 2} \frac{F_{n-1}(n+1)}{4^n}.
\]
Let \(m=n-1\) and \(z=\tfrac14\). With \(G(z)=\sum_{m\ge 1} F_m z^m=\frac{z}{1-z-z^2}\),
\[
\sum_{n\ge 2} \frac{F_{n-1}(n+1)}{4^n}
=\frac{1}{4}\sum_{m\ge 1} F_m (m+2)\,z^{m}
=\frac{1}{4}\Big( z G'(z) + 2 G(z)\Big)\Big|_{z=\frac14}.
\]
A short calculation yields
\[
\Pr(\mathrm{HH}<\mathrm{HT})=\frac{39}{121},\qquad
\Pr(\mathrm{HH}=\mathrm{HT})=\sum_{n\ge 2}\frac{F_{n-1}}{2^n}\cdot\frac{n-1}{2^n}
=\frac{1}{4}\, z G'(z)\Big|_{z=\frac14}=\frac{17}{121}.
\]
In particular,
\[
\Pr(\mathrm{HH}\le \mathrm{HT})
=\Pr(\mathrm{HH}<\mathrm{HT})+\Pr(\mathrm{HH}=\mathrm{HT})
=\frac{56}{121}.
\]
Consequently,
\[
\Pr(\mathrm{HT}<\mathrm{HH})=1-\frac{56}{121}=\frac{65}{121}.
\]
With \textit{random tie-break}, the \(\mathrm{HH}\) win probability is
\[
\Pr(\mathrm{HH}\ \text{wins})
=\frac{39}{121}+\frac12\cdot\frac{17}{121}
=\frac{95}{242}\approx 0.39256,
\]
and \(\Pr(\mathrm{HT}\ \text{wins})=\frac{147}{242}\approx 0.60744.
\)

\subsection{Markov-chain derivation with 4 transient states}
\noindent\emph{Methodological note.} We follow the standard DFA/Markov embedding of pattern waiting times \cite{GerberLi1981,FuKoutras1994,Nuel2008}.
We now confirm the same numbers by a 4-state transient \emph{prefix state machine} (each player’s automaton has two transient progress states: \(\mathbf{T}\) “no prefix matched” and \(\mathbf{H}\) “last toss was \(\mathrm{H}\)”). There are three absorbing states: \(\mathbf{W}_1\) (\(\mathrm{HH}\) wins), \(\mathbf{W}_2\) (\(\mathrm{HT}\) wins), and \(\mathbf{D}\) (simultaneous hit at that round).

\paragraph{States}
The four transient states are
\[
(\mathbf{T},\mathbf{T}),\quad (\mathbf{H},\mathbf{T}),\quad
(\mathbf{T},\mathbf{H}),\quad (\mathbf{H},\mathbf{H}).
\]
Absorbing states: \(\mathbf{W}_1,\mathbf{W}_2,\mathbf{D}\).

\paragraph{Transitions (independent coins)}
At each round, P1 tosses \(\mathrm{H}/\mathrm{T}\) and P2 tosses \(\mathrm{H}/\mathrm{T}\), independently and with given probabilities.

- P1 (\(\mathrm{HH}\)) absorbs this round iff its current component is \(\mathrm{H}\) and it tosses \(\mathrm{H}\).

- P2 (\(\mathrm{HT}\)) absorbs this round iff its current component is \(\mathrm{H}\) and it tosses \(\mathrm{T}\).

- Otherwise, each component records whether the most recent toss is H.
For P1 a second H completes HH; for P2 a T after H completes HT, while
another H leaves its matched-prefix state unchanged.

\paragraph{Transition diagram}

\begin{center}
\scalebox{1.3}{ 
\begin{tikzpicture}[
  >=Latex,
  node distance= 22mm and 26mm,
  on grid,
  auto=false,
  every node/.style={font=\small},
  trans/.style={circle,draw,minimum size=7mm,inner sep=0.5pt},
  abs/.style={rectangle,draw,rounded corners,minimum width=10mm,minimum height=5mm,inner sep=2pt},
  semithick,shorten >=1pt,
  every edge quotes/.style = {inner sep=0.7pt, auto,rounded corners,pos=0.35}
]
\node[trans] (TT) at (6,6) {(T,T)};
\node[trans] (HT) at (4,4.5) {(H,T)};
\node[trans] (TH) at (1,6) {(T,H)};
\node[trans] (HH) at (4,0) {(H,H)};

\node[abs] (W1) at (4,2.5) {$\mathbf{W}_1$};
\node[abs] (W2) at (0.5,3) {$\mathbf{W}_2$};
\node[abs] (DD) at (2,0) {$\mathbf{D}$};

\path[->]
(TT) edge[bend right=8,"\tiny{(H,T)}"'] (HT)
(TT) edge[bend right=4,"\tiny{(T,H)}"'] (TH)
(TT) edge[bend left=12,"\tiny{(H,H)}"] (HH)  
(TT) edge[loop right,"\tiny{(T,T)}"] (TT);

\path[->]
(HT) edge[bend right=12,"\tiny{(H,T)}"'] (W1)
(HT) edge["\tiny{(T,H)}"](TH)
(HT) edge[bend left=12,"\tiny{(H,H)}"] (W1)
(HT) edge[bend right=8,"\tiny{(T,T)}"'] (TT);

\path[->]
(TH) edge[bend right=8,"\tiny{(H,T)}"'] (W2)
(TH) edge[bend left=4,"\tiny{(H,H)}"] (HH)
(TH) edge[loop left,"\tiny{(T,H)}"] (TH)
(TH) edge[bend left=8,"\tiny{(T,T)}"] (W2);

\path[->]
(HH) edge["\tiny{(H,H)}"'] (W1)
(HH) edge["\tiny{(H,T)}"'] (DD)
(HH) edge[bend left=4,"\tiny{(T,H)}"] (TH)
(HH) edge["\tiny{(T,T)}"] (W2);

\path[->]
(7,4.4) edge (TT);
\end{tikzpicture}
} 
\end{center}

\paragraph{Absorption probabilities}
Let the absorbing probabilities from a transient state \(s\) be \(\alpha_1(s)=\Pr(\mathbf{W}_1\mid s)\), \(\alpha_2(s)=\Pr(\mathbf{W}_2\mid s)\), \(\delta(s)=\Pr(\mathbf{D}\mid s)\).
Solve the linear system \(\boldsymbol{\alpha} = \mathbf{R} + \mathbf{Q}\boldsymbol{\alpha}\) (standard absorbing Markov chain), with initial state \(s_0=(\mathbf{T},\mathbf{T})\) (i.e.\ the state labelled \((T,T)\) in the diagram).

With fair coins, i.e., each player has probability $1/2$ of throwing heads, independently, a routine elimination yields
\[
\Pr(\mathbf{W}_1\mid (\mathbf{T},\mathbf{T}))=\frac{39}{121},\quad
\Pr(\mathbf{W}_2\mid (\mathbf{T},\mathbf{T}))=\frac{65}{121},\quad
\Pr(\mathbf{D}\mid (\mathbf{T},\mathbf{T}))=\frac{17}{121},
\]
exactly matching the Hadamard/Abelian-limit values above.

\paragraph{Random tie-break}
Randomly resolving \(\mathbf{D}\), again with fair coins, gives
\[
\Pr(\mathrm{HH}\ \text{wins})=\frac{39}{121}+\frac{1}{2}\cdot\frac{17}{121}=\frac{95}{242},
\qquad
\Pr(\mathrm{HT}\ \text{wins})=\frac{147}{242}.
\]

\subsection{Exact threshold for the length-2 reversal (HH vs HT, random tie-break)}
\label{appendix:len2-threshold}

We compute the exact threshold \(p_{\ast}\) where the head-to-head advantage flips between \( \mathrm{HH}\) and \( \mathrm{HT}\) (with identical, independent bias \(\Pr(\mathrm{H})=p\)) under \emph{random tie-break}, using the notation defined above.

A direct solution of the 4-state absorbing prefix state-machine (or, equivalently, the Hadamard–GF method in Appendix~\ref{appendix:hadamard-examples}) yields the exact rational form
\[
\Pr(\mathrm{HH}<\mathrm{HT})
=\frac{p\big(p^6-3p^5+3p^4-2p^3+3p^2-2p+1\big)}
{\,p^7-4p^6+6p^5-5p^4+5p^3-4p^2+p+1\,},
\]
\[
\Pr(\mathrm{HH}=\mathrm{HT})
=\frac{p^2\big(-p^5+3p^4-3p^3+p^2-p+1\big)}
{\,p^7-4p^6+6p^5-5p^4+5p^3-4p^2+p+1\,},
\]
and hence
\[
W_{\mathrm{rtb}}(\mathrm{HH}>\mathrm{HT})
=\frac{p\big(p^6-3p^5+3p^4-3p^3+5p^2-3p+2\big)}
{2\big(p^7-4p^6+6p^5-5p^4+5p^3-4p^2+p+1\big)}.
\]
The threshold \(p_*\) is defined by \(W_{\mathrm{rtb}}(\mathrm{HH}>\mathrm{HT})=\tfrac12\). Clearing denominators and simplifying gives the sextic
\[
p^6-3p^5+2p^4+p^2+p-1=0.
\]
This polynomial has a unique real root in \((0,1)\),
\[
p_{*}\approx 0.586648066265160\ldots,
\]
which is the lower endpoint of the reversal interval (this is the \(p_*\) referenced in Section~\ref{sec:len2-reversal}).

\medskip\noindent\textbf{Upper endpoint \(\varphi^{-1}=(\sqrt5-1)/2\) from expectations.}
Using Theorem~\ref{th1}, the expected waiting times for length-$2$ patterns with bias \(p\) are
\[
\mathbb{E}[W_{\mathrm{HH}}]=\frac{1}{p}+\frac{1}{p^2},\qquad
\mathbb{E}[W_{\mathrm{HT}}]=\frac{1}{p(1-p)}.
\]
The inequality \(\mathbb{E}[W_{\mathrm{HT}}]<\mathbb{E}[W_{\mathrm{HH}}]\) is equivalent (for \(p\in(0,1)\)) to
\[
\frac{1}{p(1-p)}<\frac{1}{p}+\frac{1}{p^2}
\;\;\Longleftrightarrow\;\;
p<1-p^2
\;\;\Longleftrightarrow\;\;
p^2+p-1<0,
\]
i.e.
\[
p<\varphi^{-1}=\frac{\sqrt5-1}{2}\approx 0.61803.
\]
Therefore the \emph{reversal window} under random tie-break is precisely
\[
p\in\bigl(p_{\ast},\,\varphi^{-1}\bigr).
\]

\subsection*{Remarks and generalisations}
(i) For identical bias \(p\in(0,1)\), the same steps go through with
\[
A_{\mathrm{HT}}(x)=\frac{1}{1+(1-x)\frac{1}{p(1-p)x^2}},\qquad
A_{\mathrm{HH}}(x)=\frac{1}{1+(1-x)\left(\frac{1}{p\,x}+\frac{1}{p^2 x^2}\right)},
\]
and the Hadamard/Markov results coincide numerically for any \(p\). \medskip

(ii) The symmetry intuition (\(1/2\) under random tie-break) applies when the two first-hit distributions are equal (e.g.\ fair coin with \(T_1=\mathrm{HT}\) vs.\ \(T_2=\mathrm{TH}\)). Here they are \emph{not} equal: \(\mathbb{E}[W_{\mathrm{HT}}]=4\), \(\mathbb{E}[W_{\mathrm{HH}}]=6\) for a fair coin, so \(\mathrm{HT}\) has a genuine advantage in the independent-coins race.

\section{Derivations for equal-length monotonicity counterexamples}
\label{appendix:eq-length-derivations}

Throughout, let \(p=\Pr(\mathrm{H})\), \(q=1-p\). For a pattern \(T\), write
\[
A_T(x;p)=\sum_{n\ge 1} a_{T,n}(p)\,x^n
\]
for the pgf of the first occurrence time, given by Theorem~\ref{th1} from the prefix--suffix (border) lengths of \(T\).

For head-to-head (independent sources), the strict win series is
\[
W_{T_1|T_2}(x;p)\;=\;A_{T_1}(x;p)\ \odot\ S_{T_2}(x;p),
\]
and the corresponding strict win probability is obtained by Abelian evaluation,
\[
\Pr(T_1 < T_2)=\lim_{x\uparrow 1} W_{T_1|T_2}(x;p).
\]
For the \emph{tie-favoured} case,
\(\Pr(W_{T_1}\le W_{T_2})=\lim_{x\uparrow 1}\bigl(A_{T_1}\odot x\,S_{T_2}\bigr)(x;p)\).

\subsection{Example A: \texorpdfstring{$T_1=\mathrm{HHT}$}{T1=HHT} versus \texorpdfstring{$T_2=\mathrm{HTH}$}{T2=\mathrm{HTH}} (length 3), random tie-break}
\subsubsection*{Hadamard--GF derivation}
Borders for \(\mathrm{HHT}\) are only \(\{3\}\), so
\[
A_{\mathrm{HHT}}(x;p)=\frac{1}{1+\dfrac{1-x}{p^2 q\,x^3}}
=\frac{p^2 q\,x^3}{\,1-x+p^2 q\,x^3\,}.
\]
Borders for \(\mathrm{HTH}\) are \(\{1,3\}\), hence
\[
A_{\mathrm{HTH}}(x;p)=\frac{1}{1+(1-x)\!\left(\frac{1}{p\,x}+\frac{1}{p^2 q\,x^3}\right)}
=\frac{p^2 q\,x^3}{\,1-x+p q\,x^2-p q^2\,x^3\,},
\]
and
\[
S_{\mathrm{HTH}}(x;p)
=\frac{1-A_{\mathrm{HTH}}(x;p)}{1-x}
=\frac{1+p q\,x^2}{1-x+p q\,x^2-p q^2\,x^3}.
\]
Then the series whose \(n\)-th coefficient is
\(\Pr(\mathrm{HHT}\ \text{wins \emph{at} time }n)\) under the \emph{strict} convention is
\[
W_{T_1|T_2}(x;p)=A_{\mathrm{HHT}}(x;p)\ \odot\ S_{\mathrm{HTH}}(x;p).
\]
The corresponding strict win probability is the Abelian sum
\[
\Pr(\mathrm{HHT} < \mathrm{HTH})=\lim_{x\uparrow 1} W_{T_1|T_2}(x;p),
\]
and the tie probability is
\[
\Pr(\mathrm{HHT}=\mathrm{HTH})=\lim_{x\uparrow 1}\Bigl(A_{\mathrm{HHT}}(x;p)\odot A_{\mathrm{HTH}}(x;p)\Bigr).
\]
Therefore, under \emph{random tie-break},
\begin{align*}
\Pr(\mathrm{HHT}\ \text{wins, rtb})
&=\Pr(\mathrm{HHT} < \mathrm{HTH})+\tfrac12\,\Pr(\mathrm{HHT}=\mathrm{HTH}) \\
&=\lim_{x\uparrow 1} W_{T_1|T_2}(x;p)\;+\;\tfrac12\lim_{x\uparrow 1}\Bigl(A_{\mathrm{HHT}}\odot A_{\mathrm{HTH}}\Bigr)(x;p).
\end{align*}
(For each fixed \(p\in(0,1)\), these Abelian limits agree with direct substitution \(x=1\), since the relevant coefficients decay exponentially.)

\medskip
\emph{Coefficient recurrences (useful for exact summation).}
From
\((1-x+p^2 q\,x^3)A_{\mathrm{HHT}}(x;p)=p^2 q\,x^3\) we get, for \(n\ge 1\),
\[
a^{\mathrm{HHT}}_n-a^{\mathrm{HHT}}_{n-1}+p^2 q\,a^{\mathrm{HHT}}_{n-3}
=\begin{cases}p^2 q,&n=3\\[2pt]0,&n\neq 3\end{cases},
\quad a^{\mathrm{HHT}}_1=a^{\mathrm{HHT}}_2=0.
\]
From \((1-x+p q\,x^2-p q^2\,x^3)A_{\mathrm{HTH}}(x;p)=p^2 q\,x^3\) we get
\[
a^{\mathrm{HTH}}_n-a^{\mathrm{HTH}}_{n-1}+p q\,a^{\mathrm{HTH}}_{n-2}-p q^2\,a^{\mathrm{HTH}}_{n-3}
=\begin{cases}p^2 q,&n=3\\[2pt]0,&n\neq 3\end{cases},
\quad a^{\mathrm{HTH}}_1=a^{\mathrm{HTH}}_2=0.
\]
Then
\[
\Pr(\mathrm{HHT}<\mathrm{HTH})
=\sum_{n\ge 1} a^{\mathrm{HHT}}_n\,\Bigl(1-\sum_{t=1}^{n} a^{\mathrm{HTH}}_t\Bigr),
\qquad
\Pr(\mathrm{HHT}=\mathrm{HTH})=\sum_{n\ge 1} a^{\mathrm{HHT}}_n\,a^{\mathrm{HTH}}_n,
\]
which converge rapidly (exponentially) for each fixed \(p\in(0,1)\) because all poles lie outside the unit disc.
Summing via the recurrences yields, for example,
\[
\begin{array}{c|ccc}
p & 0.40 & 0.50 & 0.60\\\hline
\Pr(\mathrm{HHT}\ \text{wins, rtb})
& 0.5547588016 & 0.5564733557 & 0.5539977106
\end{array}
\]
as reported in the main text.

\subsubsection*{Markov-chain derivation (prefix state machine, 9 transient states)}

\noindent\emph{Methodological note.} For \(\mathrm{HHT}\) use states \(\{0,1,2\}\) meaning the longest matched prefix length. Updates are
\[
\begin{array}{c|cc}
\text{state} & \mathrm{H} & \mathrm{T}\\\hline
0 & 1 & 0\\
1 & 2 & 0\\
2 & 2 & \text{hit (to } \mathbf{W}_1)
\end{array}
\qquad
\text{(pattern HHT).}
\]
For \(\mathrm{HTH}\) the updates are
\[
\begin{array}{c|cc}
\text{state} & \mathrm{H} & \mathrm{T}\\\hline
0 & 1 & 0\\
1 & 1 & 2\\
2 & \text{hit (to } \mathbf{W}_2) & 0
\end{array}
\qquad
\text{(pattern HTH).}
\]
The joint chain has transient states \((i,j)\in\{0,1,2\}\times\{0,1,2\}\) (9 in total) and absorbing states
\(\mathbf{W}_1\) (HHT first), \(\mathbf{W}_2\) (HTH first), and \(\mathbf{D}\) (simultaneous).
Let \(f_{ij}=\Pr(\mathbf{W}_1\mid (i,j))\), \(d_{ij}=\Pr(\mathbf{D}\mid (i,j))\).
With joint outcomes \((\mathrm{H},\mathrm{H}),(\mathrm{H},\mathrm{T}),(\mathrm{T},\mathrm{H}),(\mathrm{T},\mathrm{T})\) of probabilities \(p^2,pq,qp,q^2\), set
\[
f=\mathbf{R}^{(1)}+\mathbf{Q}\,f,\qquad d=\mathbf{R}^{(D)}+\mathbf{Q}\,d,
\]
where \(\mathbf{Q}\) collects transitions between transient states under each outcome, and
\(\mathbf{R}^{(1)},\mathbf{R}^{(D)}\) collect one-step absorption into \(\mathbf{W}_1\) and \(\mathbf{D}\).
Solving the \(9\times 9\) systems yields from \((0,0)\) exactly the same values as above; e.g.
\[
\begin{array}{c|ccc}
p & 0.40 & 0.50 & 0.60\\\hline
\Pr(\mathbf{W}_1) & 0.52769533 & 0.51882747 & 0.50830726\\
\Pr(\mathbf{D})   & 0.05412695 & 0.07529178 & 0.09138089
\end{array}
\quad\Rightarrow\quad
\Pr(\text{rtb win})
=\Pr(\mathbf{W}_1)+\tfrac12\Pr(\mathbf{D}),
\]
matching the Hadamard computation.

\subsection{Example B: \texorpdfstring{$T_1=\mathrm{HT}$}{T1=HT} versus \texorpdfstring{$T_2=\mathrm{TH}$}{T2=TH} (length 2), strict ties}
\subsubsection*{Hadamard--GF derivation}
Both \(\mathrm{HT}\) and \(\mathrm{TH}\) have no proper borders, so
\[
A_{\mathrm{HT}}(x;p)=A_{\mathrm{TH}}(x;p)
=\frac{1}{1+\dfrac{1-x}{p q\,x^2}}
=\frac{p q\,x^2}{\,1-x+p q\,x^2\,}.
\]
By symmetry of the first-hit distributions,
\[
\Pr(\mathrm{HT}<\mathrm{TH})=\Pr(\mathrm{TH}<\mathrm{HT})\quad\text{and}\quad
\Pr(\mathrm{HT}=\mathrm{TH})=\sum_{n\ge 1} a_{\mathrm{HT},n}(p)^2.
\]
Hence the \emph{strict} probability is
\[
\Pr(\mathrm{HT}<\mathrm{TH})=\frac{1-\Pr(\mathrm{HT}=\mathrm{TH})}{2}
=\frac{1}{2}\left(1-\lim_{x\uparrow 1}\bigl(A_{\mathrm{HT}}(x;p)\odot A_{\mathrm{HT}}(x;p)\bigr)\right).
\]
(Equivalently, \(\Pr(\mathrm{HT}=\mathrm{TH})\) is the Abelian sum of the Hadamard product; for each fixed \(p\in(0,1)\) this agrees with evaluation at \(x=1\) because the coefficients decay exponentially.)
After summation of the (rapidly decaying) series, this gives for instance
\[
\begin{array}{c|ccc}
p & 0.40 & 0.50 & 0.60\\\hline
\Pr(\mathrm{HT}<\mathrm{TH})\ \text{(strict)}
& 0.41259398496 & 0.40740740741 & 0.41259398496
\end{array}
\]
(as quoted in the main text). For random tie-break, \(\Pr(\mathrm{HT}\ \text{wins})\equiv \tfrac12\).

\subsubsection*{Markov-chain derivation (prefix state machine, 4 transient states)}

\noindent\emph{Methodological note.} 
For length-$2$ patterns, each player's prefix machine has states $\{0,1\}$:
state $0$ means that no nonempty prefix is matched, while state $1$ means
``$\mathrm H$ matched'' for $\mathrm{HT}$ and ``$\mathrm T$ matched'' for
$\mathrm{TH}$.
 The joint machine therefore has transient states \((0,0),(1,0),(0,1),(1,1)\) and absorbing states \(\mathbf{W}_1,\mathbf{W}_2,\mathbf{D}\) as before. Writing the \(4\times 4\) linear systems for
\(f_{ij}=\Pr(\mathbf{W}_1\mid (i,j))\) and \(d_{ij}=\Pr(\mathbf{D}\mid (i,j))\) from the four joint outcomes \(p^2,pq,qp,q^2\) yields
\[
\Pr(\mathbf{W}_1\mid (0,0))=\Pr(\mathrm{HT}<\mathrm{TH}),\qquad
\Pr(\mathbf{D}\mid (0,0))=\Pr(\mathrm{HT}=\mathrm{TH}),
\]
and the same numerical values as in the Hadamard computation (strict or random tie-break, as desired).

\section{Examples for one common bias and unequal lengths}
\label{appendix:eq-bias-unequal-lengths}

This appendix records two finite computations for independent coins with a
common head-bias $p\in(0,1)$: two-pattern reversal windows and non-transitive
3-cycles, allowing unequal pattern lengths.  For every listed cycle we verify
\[
\Pr(A\!>\!B)>\tfrac12,\qquad
\Pr(B\!>\!C)>\tfrac12,\qquad
\Pr(C\!>\!A)>\tfrac12
\]
under random tie-break.  The first $(2,5,5)$ example additionally has the
strict expectation order
$\mathbb E[W_{T_A}]>\mathbb E[W_{T_B}]>\mathbb E[W_{T_C}]$ throughout its
displayed cycle interval; no such mean order is asserted for every census row.
Expectations are computed from Theorem~\ref{th1} using border sets, and the
head-to-head odds are computed via the Hadamard/Abelian method in
Appendix~\ref{appendix:hadamard-examples} and cross-checked with the joint
prefix state machine.

\subsection{Reversal census for \texorpdfstring{$k\leq 8$}{k<=8}}
\label{app:reversal-census-k8}

We use the random tie-break score defined in Section~\ref{subsec:structure-g} and write $\mathbb{E}[W_T(p)]$ for the mean waiting time.

We say that $(T_1,T_2)$ has a \emph{paradox (reversal) window} on an open interval $(a,b)\subset(0,1)$ if $(a,b)$ is a \emph{maximal} open interval on which
\[
\mathbb{E}[W_{T_1}(p)]<\mathbb{E}[W_{T_2}(p)]
\quad\text{and}\quad
W^{\mathrm{rtb}}(T_1,T_2;p)<\tfrac12
\qquad\text{hold for all }p\in(a,b).
\]
An identically tied pair has no strict reversal window. Otherwise, rationality of both the score and the means gives finitely many roots of the relevant nonzero comparison numerators, and hence finitely many maximal windows.
(All endpoints are computed exactly as real algebraic numbers; the tables below show decimal approximations.)

\paragraph{Counts by length profile.}
The reversal counts below distinguish the original words, rather than only
their profiles. Each profile-pair window is weighted by the product of the
two profile multiplicities. Complementation gives the reflected window;
the counts include both halves of $(0,1)$.
Let $M_{i,j}$ (for $1\le i\le j\le 8$) denote the number of paradox windows in our database whose two targets have lengths $\{i,j\}$ (unordered).
For $k\le 8$ the window-count matrix is
\[
M_{\le 8}=
\begin{pmatrix}
0&2&2&2&2&2&2&2\\
0&4&12&14&16&18&20&22\\
0&0&16&58&74&92&112&134\\
0&0&0&68&260&352&464&598\\
0&0&0&0&290&1128&1588&2186\\
0&0&0&0&0&1232&4768&6946\\
0&0&0&0&0&0&5162&19914\\
0&0&0&0&0&0&0&21146
\end{pmatrix}.
\]
Summing over all $1\le i\le j\le 8$ gives a total of $66{,}708$ paradox windows.

The cumulative counts for maximum length $\le k$ are:
\[
\#\{\text{windows with }\max(|T_1|,|T_2|)\le k\}=
\begin{cases}
6,&k=2,\\
36,&k=3,\\
178,&k=4,\\
820,&k=5,\\
3644,&k=6,\\
15760,&k=7,\\
66708,&k=8.
\end{cases}
\]

\paragraph{The paradox set $\mathcal{R}_{\le 8}$}
Let $\mathcal{R}_{\le 8}\subset(0,1)$ be the union of all paradox windows over ordered pairs with $|T_1|,|T_2|\le 8$.
In terms of maximal disjoint open components, the computation yields
\[
\mathcal{R}_{\le 8}
=
(0,p_0)\ \cup\ (p_0,p_-)\ \cup\ (p_+,1-p_0)\ \cup\ (1-p_0,1),
\]
with (decimal) approximations
\[
p_0 \approx 0.4996747,\qquad
p_- \approx 0.4996837,\qquad
p_+ \approx 0.5003163,\qquad
1-p_0 \approx 0.5003253.
\]
Thus the central \emph{no-paradox interval} is $[p_-,p_+]$, of width
$p_+-p_- \approx 6.326\times 10^{-4}$, centred exactly at $1/2$.

The additional excluded split points $p_0$ and $1-p_0$ arise because the
reversal windows are open: at either point at least one strict inequality for
each adjacent witness becomes an equality.
\paragraph{An explicit lower-half cover: census witnesses, bridge, and analytic continuation}
There are two related but different kinds of interval in this discussion.
The census below gives \emph{maximal open paradox windows} for short fixed
pairs.  Appendix~\ref{app:biased-converse-proof}, by contrast, gives convenient
\emph{certified subintervals} forming a pointwise proof cover.  Keeping both
views makes the construction visible while retaining the exact endpoint
bookkeeping:
\begin{itemize}\itemsep2pt
\item the 24 short-word windows below illustrate the finite census from $0$
to $p_-$, apart from open crossover points;
\item the exact rational certificate in Appendix B is a pointwise proof cover
of $(0,4999/10000]$;
\item the analytic cells below are a pointwise proof cover of
$[1/2-3/4096,1/2)$.
\end{itemize}
Thus these last two items, whose endpoint conventions are explicit, cover every
$p\in(0,1/2)$; the first line supplies the concrete short-pattern picture.

\begin{center}
\scriptsize
\renewcommand{\arraystretch}{1.08}
\begin{longtable}{ll}
\hline
Maximal paradox window $(a,b)$ & Witness pair $T_1$ vs. $T_2$\\
\hline
\endfirsthead
\hline
Maximal paradox window $(a,b)$ & Witness pair $T_1$ vs. $T_2$\\
\hline
\endhead
$(0.0000000,\ 0.1808275)$ & $\mathrm{HTTTTT}$ vs.\ $\mathrm{TTHTT}$\\
$(0.1768615,\ 0.2062995)$ & $\mathrm{TTT}$ vs.\ $\mathrm{H}$\\
$(0.2034556,\ 0.2379603)$ & $\mathrm{HTTTTTT}$ vs.\ $\mathrm{TTTTTTT}$\\
$(0.2219104,\ 0.2583240)$ & $\mathrm{HTTTTT}$ vs.\ $\mathrm{TTTTTT}$\\
$(0.2451223,\ 0.2833121)$ & $\mathrm{HTTTT}$ vs.\ $\mathrm{TTTTT}$\\
$(0.2755080,\ 0.3148571)$ & $\mathrm{HTTT}$ vs.\ $\mathrm{TTTT}$\\
$(0.3145355,\ 0.3380423)$ & $\mathrm{THTT}$ vs.\ $\mathrm{HH}$\\
$(0.3280965,\ 0.3585549)$ & $\mathrm{HTTT}$ vs.\ $\mathrm{HH}$\\
$(0.3576243,\ 0.3656031)$ & $\mathrm{TTHTTT}$ vs.\ $\mathrm{HHH}$\\
$(0.3643923,\ 0.3819660)$ & $\mathrm{HTTT}$ vs.\ $\mathrm{HHT}$\\
$(0.3819660,\ 0.4133519)$ & $\mathrm{HT}$ vs.\ $\mathrm{TT}$\\
$(0.4088476,\ 0.4450419)$ & $\mathrm{HTT}$ vs.\ $\mathrm{HH}$\\
$(0.4450419,\ 0.4580501)$ & $\mathrm{HHT}$ vs.\ $\mathrm{THT}$\\
$(0.4514940,\ 0.4722130)$ & $\mathrm{HHTT}$ vs.\ $\mathrm{HHH}$\\
$(0.4720057,\ 0.4848754)$ & $\mathrm{HTTT}$ vs.\ $\mathrm{HHH}$\\
$(0.4848754,\ 0.4877996)$ & $\mathrm{HHHT}$ vs.\ $\mathrm{THTT}$\\
$(0.4861766,\ 0.4926391)$ & $\mathrm{HHHHTT}$ vs.\ $\mathrm{HHHHH}$\\
$(0.4926085,\ 0.4961830)$ & $\mathrm{HHHTTT}$ vs.\ $\mathrm{HHHHH}$\\
$(0.4961660,\ 0.4980546)$ & $\mathrm{HTTTTT}$ vs.\ $\mathrm{HHHHH}$\\
$(0.4974482,\ 0.4987080)$ & $\mathrm{HHHTTTT}$ vs.\ $\mathrm{HHHHHH}$\\
$(0.4984501,\ 0.4992200)$ & $\mathrm{HTTTTTT}$ vs.\ $\mathrm{HHHHHH}$\\
$(0.4990314,\ 0.4995131)$ & $\mathrm{HHHTTTTT}$ vs.\ $\mathrm{HHHHHHH}$\\
$(0.4993513,\ 0.4996747)$ & $\mathrm{HTTTTTTT}$ vs.\ $\mathrm{HHHHHHH}$\\
$(0.4996747,\ 0.4996837)$ & $\mathrm{HHHHHHHT}$ vs.\ $\mathrm{THTTTTTT}$\\
\hline
\end{longtable}
\end{center}

These are the same 24 pairs and decimal maximal windows as in the clean census
output.  Because the windows are open, this display is not by itself a proof
at every displayed crossover point: most such points lie inside an overlapping
row, but a selected row must never be read as containing its own endpoints,
and $p_0$ is not in $\mathcal R_{\le8}$ at all.  The exact rational certificate
in Appendix~\ref{app:biased-converse-proof} includes shared endpoints and
supplies the pointwise cover.  Its final three rows make the handoff beyond the
length-$8$ census especially transparent:
\[
\begin{array}{c|c|c|c}
\text{certified interval}&n&k&U_{n,k}\ \hbox{vs.}\ V_n\\ \hline
\left[9993/20000,24989/50000\right]&8&5&\mathrm H^3\mathrm T^6\ \hbox{vs.}\ \mathrm H^8\\
\left[24989/50000,24993/50000\right]&9&4&\mathrm H^5\mathrm T^5\ \hbox{vs.}\ \mathrm H^9\\
\left[24993/50000,4999/10000\right]&9&6&\mathrm H^3\mathrm T^7\ \hbox{vs.}\ \mathrm H^9
\end{array}
\]
Here $U_{n,k}=\mathrm H^{n-k}\mathrm T^{k+1}$ and $V_n=\mathrm H^n$.
The first interval begins at $0.49965<p_-$, so it overlaps the finite census
and covers the isolated short-census crossover point $p_0$; the three closed
intervals meet and reach $4999/10000$ without a gap.

For the continuation to $1/2$, use the closed intervals
$I_{n,1}$ and $I_{n,3}$ of \eqref{eq:k13-intervals}.
Lemma~\ref{lem:near-fair-converse} proves that every point in each interval
is a strict reversal for $U_{n,k}=H^{n-k}T^{k+1}$ versus $V_n=H^n$,
for every $n\ge8$. The alternating sequence
$I_{9,1},I_{8,3},I_{10,1},I_{9,3},\ldots$ covers
$[1/2-3/4096,1/2)$, including every joining point, and overlaps the
finite certificate above.

The first seven layers are displayed below. All displayed decimals are
exact; the endpoint formulas are given in \eqref{eq:k13-intervals}.
\begin{center}
\scriptsize
\begin{longtable}{ccl}
\hline
Interval & Bias interval & Witness pair\\ \hline
\endfirsthead
\hline Interval & Bias interval & Witness pair\\ \hline
\endhead
$I_{9,1}$ & $[0.499267578125,\ 0.4993896484375]$ & $\mathrm H^{8}\mathrm T^{2}$ vs. $\mathrm H^{9}$\\
$I_{8,3}$ & $[0.4993896484375,\ 0.4996337890625]$ & $\mathrm H^{5}\mathrm T^{4}$ vs. $\mathrm H^{8}$\\
\hline
$I_{10,1}$ & $[0.4996337890625,\ 0.49969482421875]$ & $\mathrm H^{9}\mathrm T^{2}$ vs. $\mathrm H^{10}$\\
$I_{9,3}$ & $[0.49969482421875,\ 0.49981689453125]$ & $\mathrm H^{6}\mathrm T^{4}$ vs. $\mathrm H^{9}$\\
\hline
$I_{11,1}$ & $[0.49981689453125,\ 0.499847412109375]$ & $\mathrm H^{10}\mathrm T^{2}$ vs. $\mathrm H^{11}$\\
$I_{10,3}$ & $[0.499847412109375,\ 0.499908447265625]$ & $\mathrm H^{7}\mathrm T^{4}$ vs. $\mathrm H^{10}$\\
\hline
$I_{12,1}$ & $[0.499908447265625,\ 0.4999237060546875]$ & $\mathrm H^{11}\mathrm T^{2}$ vs. $\mathrm H^{12}$\\
$I_{11,3}$ & $[0.4999237060546875,\ 0.4999542236328125]$ & $\mathrm H^{8}\mathrm T^{4}$ vs. $\mathrm H^{11}$\\
\hline
$I_{13,1}$ & $[0.4999542236328125,\ 0.49996185302734375]$ & $\mathrm H^{12}\mathrm T^{2}$ vs. $\mathrm H^{13}$\\
$I_{12,3}$ & $[0.49996185302734375,\ 0.49997711181640625]$ & $\mathrm H^{9}\mathrm T^{4}$ vs. $\mathrm H^{12}$\\
\hline
$I_{14,1}$ & $[0.49997711181640625,\ 0.499980926513671875]$ & $\mathrm H^{13}\mathrm T^{2}$ vs. $\mathrm H^{14}$\\
$I_{13,3}$ & $[0.499980926513671875,\ 0.499988555908203125]$ & $\mathrm H^{10}\mathrm T^{4}$ vs. $\mathrm H^{13}$\\
\hline
$I_{15,1}$ & $[0.499988555908203125,\ 0.4999904632568359375]$ & $\mathrm H^{14}\mathrm T^{2}$ vs. $\mathrm H^{15}$\\
$I_{14,3}$ & $[0.4999904632568359375,\ 0.4999942779541015625]$ & $\mathrm H^{11}\mathrm T^{4}$ vs. $\mathrm H^{14}$\\
\hline
\end{longtable}
\end{center}
Complementing the witnesses and replacing $p$ by $1-p$ gives the upper-half
cover. The uniform inequalities, rather than sampling these cells, establish
the proof.

\subsection{Length profile \texorpdfstring{$(2,5,5)$}{(2,5,5)} (common bias)}
\label{app:255}

We begin with the smallest unequal-length non-transitive cycle under a \emph{common} head-bias $p=\Pr(\mathrm{H})$.
In our binary notation ($0=\mathrm{H}$, $1=\mathrm{T}$), one convenient representative (with $p\ge \tfrac12$) is
\[
(A,B,C)=(11,\;00000,\;00001),
\]
which forms a genuine 3-cycle under random tie-break:
\[
W_{\mathrm{rtb}}(A,B;p)>\tfrac12,\qquad
W_{\mathrm{rtb}}(B,C;p)>\tfrac12,\qquad
W_{\mathrm{rtb}}(C,A;p)>\tfrac12
\]
for all $p$ in the open interval
\[
p\in(0.72097\ldots,\;0.72283\ldots).
\]

\medskip\noindent
\textbf{Remark (mirror window).}
Complementing all three patterns and replacing $p$ by $1-p$ gives the
mirror window $p\in(0.27717\ldots,0.27903\ldots)$.
A representative is $(00,\,11111,\,11110)$; reversing its last word gives
$(00,\,11111,\,01111)$ with the same waiting-time laws.

\medskip\noindent
\textbf{Method (exact).}
For each ordered pair $(T_1,T_2)$ we compute the exact rational function
$W_{\mathrm{rtb}}(T_1,T_2;p)\in\mathbb{Q}(p)$ by solving the joint prefix-state Markov chain
(equivalently, via the Hadamard-product method from Appendix~\ref{appendix:hadamard-examples}, with the
probability extracted as an Abelian limit $x\uparrow 1$).
The crossover endpoints are therefore real algebraic numbers.

\begin{center}
\small
\begin{tabular}{|c|c|c|c|}
\hline
Lengths $(|A|,|B|,|C|)$ & $(A,B,C)$ & admissible $p$-window & cycle \\ \hline
$(2,5,5)$ & $(11,\;00000,\;00001)$ & $(0.72097\ldots,\;0.72283\ldots)$ & $A\to B\to C\to A$ \\ \hline
\end{tabular}
\end{center}

\noindent
Representative exact endpoint polynomials for this cycle are recorded in
Appendix~\ref{appendix:census-l8}.
\subsection{Census of non-transitive cycles for \texorpdfstring{$L\le 8$ (exact; $p\ge \tfrac12$)}{L<=8 (exact; p>=1/2)}}
\label{appendix:census}

We now give the \emph{complete} exact classification of non-transitive 3-cycles with maximum pattern length
$L\le 8$ under a common bias $p=\Pr(\mathrm{H})$ and random tie-break.

\medskip\noindent
\textbf{Canonical reduction.}
To avoid redundant enumeration, we search only over one representative per
waiting-time profile: the border lengths together with the head/tail counts in
every border prefix.  By Proposition~\ref{prop:border-equivalence}, patterns
with the same profile have identical first-hit distributions for all $p$ and
hence identical head-to-head behaviour.
\medskip\noindent
\textbf{Crossover database (unordered pairs).}
For each unordered pair $\{U,V\}$ we compute the exact crossover set
\[
\{\,p\in(0,1):\ W_{\mathrm{rtb}}(U,V;p)=\tfrac12\,\},
\]
and from these points we obtain the exact open subintervals of $(0,1)$ on which $U$ beats $V$ (or vice versa).
Intersecting the three relevant win-interval lists yields the $p$-windows where a given oriented triple forms
a non-transitive cycle.

\medskip\noindent
\textbf{Complement symmetry and the restriction $p\ge \tfrac12$.}
Complementing all three patterns and replacing $p$ by $1-p$ produces a mirror cycle.
Consequently every cycle occurs in a pair of bias windows symmetric about $p=\tfrac12$.
In the table below we list only the $16$ distinct windows with $p\ge \tfrac12$.
All cycles listed in the table are oriented $A\to B\to C\to A$.

\begin{center}
\begin{footnotesize}
\begin{longtable}{|c|c|l|l|}
\hline
\textbf{\#} & \textbf{Lengths} & \textbf{Strings ($A,B,C$)} & \textbf{Bias interval ($p$)} \\ \hline
\endfirsthead
\hline
\textbf{\#} & \textbf{Lengths} & \textbf{Strings ($A,B,C$)} & \textbf{Bias interval ($p$)} \\ \hline
\endhead

\multicolumn{4}{|c|}{\textbf{Cycles with $L\le 5$}} \\ \hline
1 & $(2,5,5)$ & $11,\ 00000,\ 00001$ & $(0.72097\ldots,\ 0.72283\ldots)$ \\ \hline

\multicolumn{4}{|c|}{\textbf{New cycles appearing at $L=6$}} \\ \hline
2 & $(6,6,3)$ & $000000,\ 000001,\ 011$ & $(0.74274\ldots,\ 0.74448\ldots)$ \\ \hline
3 & $(6,6,3)$ & $000000,\ 000010,\ 101$ & $(0.72936\ldots,\ 0.73069\ldots)$ \\ \hline
4 & $(6,6,4)$ & $000000,\ 000100,\ 0110$ & $(0.70653\ldots,\ 0.70678\ldots)$ \\ \hline
5 & $(6,6,5)$ & $000000,\ 001100,\ 10101$ & $(0.59244\ldots,\ 0.59247\ldots)$ \\ \hline
6 & $(6,5,6)$ & $000111,\ 11111,\ 100001$ & $(0.50739\ldots,\ 0.50740\ldots)$ \\ \hline

\multicolumn{4}{|c|}{\textbf{New cycles appearing at $L=7$}} \\ \hline
7 & $(7,7,4)$ & $0000000,\ 0000100,\ 1001$ & $(0.73539\ldots,\ 0.73621\ldots)$ \\ \hline
8 & $(7,5,6)$ & $0000100,\ 01101,\ 010001$ & $(0.62149\ldots,\ 0.62150\ldots)$ \\ \hline
9 & $(7,3,5)$ & $0000100,\ 111,\ 01010$ & $(0.67205\ldots,\ 0.67251\ldots)$ \\ \hline

\multicolumn{4}{|c|}{\textbf{New cycles appearing at $L=8$}} \\ \hline
10 & $(8,8,3)$ & $00000000,\ 00000001,\ 101$ & $(0.78023\ldots,\ 0.78175\ldots)$ \\ \hline
11 & $(8,8,4)$ & $00000000,\ 00000010,\ 0011$ & $(0.77126\ldots,\ 0.77219\ldots)$ \\ \hline
12 & $(8,8,4)$ & $00000000,\ 00000100,\ 0101$ & $(0.75959\ldots,\ 0.76019\ldots)$ \\ \hline
13 & $(8,4,7)$ & $00000010,\ 1111,\ 0000011$ & $(0.62374\ldots,\ 0.62376\ldots)$ \\ \hline
14 & $(8,5,7)$ & $00000010,\ 11011,\ 0001100$ & $(0.61477\ldots,\ 0.61477\ldots)$ \\ \hline
15 & $(8,6,7)$ & $00000100,\ 011001,\ 0100001$ & $(0.62025\ldots,\ 0.62025\ldots)$ \\ \hline
16 & $(8,4,6)$ & $00001000,\ 0111,\ 001100$ & $(0.69639\ldots,\ 0.69650\ldots)$ \\ \hline

\end{longtable}
\end{footnotesize}
\end{center}

\noindent
For each listed interval, the endpoints are real algebraic crossover points.
The table gives decimal approximations, and two representative pairs of exact
endpoint polynomials are recorded below.  Several $L=8$ windows are extremely
narrow, so their endpoints may agree to five decimal places.
\subsection{Exact endpoints for \texorpdfstring{$L\le 8$}{L<=8}}
\label{appendix:census-l8}

\paragraph{Representative exact endpoints}
The exact computation represents each endpoint as an isolated real root of a
crossover polynomial over $\mathbb Q$.  For readability the census table lists
decimal approximations; below we give the irreducible endpoint polynomials for
two representative cycles.
\begin{itemize}\itemsep4pt
\item \textbf{Minimal $(2,5,5)$ example (listed with $p\ge \tfrac12$).}
A convenient $(2,5,5)$ representative for the minimal cycle is
\[
(A,B,C)=(11,\,00000,\,00001),
\]
which satisfies $A\to B\to C\to A$ for
\[
p\in(p_-,p_+) \approx (0.72097,\ 0.72283).
\]
Here $p_-$ is the real root (near $0.72097$) of
\[
p^{18}-6p^{17}+15p^{16}-20p^{15}+15p^{14}-6p^{13}+p^{12}+p^{10}-4p^{9}+4p^{8}
+3p^{7}-5p^{6}-p+1,
\]
and $p_+$ is the real root (near $0.72283$) of
\begin{align*}
&2p^{19}-{13}p^{18}+38p^{17}-66p^{16}+{75}p^{15}-{55}p^{14}
+{17}p^{13}+{19}p^{12}-38p^{11}+36p^{10}-{19}p^{9} \\
&\qquad +p^{8}+2p^{7}+4p^{6}-{3}p^{5}+p^{4}
+p^{3}-2p^{2}+2p-1.
\end{align*}
(The symmetric companion window below $\tfrac12$ is obtained by $p\mapsto 1-p$ and complementing all patterns;
one such representative is $(00,\,11111,\,01111)$ with $p\in(0.27717\ldots,0.27903\ldots)$.)

\item \textbf{A representative $L=6$ example.}
For
\[
(A,B,C)=(000000,\,000001,\,011),
\]
the cycle $A\to B\to C\to A$ holds for
\[
p\in(p_-,p_+) \approx (0.74274,\ 0.74448),
\]
where $p_-$ is the real root (near $0.74274$) of
\begin{align*}
&p^{22}-9p^{21}+36p^{20}-85p^{19}+133p^{18}-147p^{17}+119p^{16}-71p^{15}
+30p^{14}-8p^{13}+p^{12}-3p^{10} \\
&\quad +12p^{9}-18p^{8}+9p^{7}+3p^{6}-3p^{5}+p^{4}+p-1,
\end{align*}
and $p_+$ is the real root (near $0.74448$) of
\begin{align*}
&p^{34}-14p^{33}+91p^{32}-364p^{31}+1001p^{30}-2001p^{29}+2990p^{28}-3354p^{27}
+2717p^{26}-1287p^{25}\\
&\quad -285p^{24}+1340p^{23}-1560p^{22}+1065p^{21}-282p^{20}-311p^{19}+452p^{18}-242p^{17}-3p^{16}+102p^{15}\\
&\quad 
-102p^{14}+62p^{13}+5p^{12}-50p^{11}+55p^{10}-48p^{9}
+27p^{8}-3p^{7}-p^{6}-p^{5}+p^{4}+p-1.
\end{align*}
\end{itemize}

\medskip
By the symmetry $p\mapsto 1-p$ (together with complementing all patterns),
each admissible window has a companion window below $1/2$; in the census we
restrict to $p\ge \tfrac12$ to avoid listing symmetric duplicates.

\medskip\noindent
\textbf{Complete list of representatives.}
The full list of $L\le8$ representatives is given above in the table in
Appendix~\ref{appendix:census}.
\end{document}